\documentclass[reqno]{amsart}

\usepackage{amssymb} 

\numberwithin{equation}{section}

\def\L{\mathfrak{L}}
\def\X{\mathfrak{X}}

\newtheorem{Corollary}{\sc Corollary}[section]
\newtheorem{Theorem}{\sc Theorem}[section]
\newtheorem{Lemma}{\sc Lemma}[section]
\newtheorem{Proposition}{\sc Proposition}[section]
\newtheorem{Remark}{\sc Remark}[section]

\newcommand{\R}{\mathbb R}
\newcommand{\C}{\mathbb C}
\newcommand{\N}{\mathbb N}

\newcommand{\F}{\mathcal{F}}

\title[The Fowler Equation]
{Existence of Travelling-Wave Solutions and Local Well-Posedness of the Fowler Equation}

\author[Borys Alvarez-Samaniego and Pascal Azerad]{Borys Alvarez-Samaniego \and Pascal Azerad}



\email{balvarez@math.uic.edu} 
\email{azerad@math.univ-montp2.fr}

\begin{document}
\maketitle 
 {\footnotesize \centerline{Institut de Math\'ematiques et Mod\'elisation de Montpellier (I3M)-UMR 5149 CNRS}
  \centerline{Universit\'e Montpellier 2} 
  \centerline{CC 051, Place Eug\`ene Bataillon, 34095 Montpellier Cedex 5, France}}

\begin{abstract}
We study the existence of travelling-waves and local well-posedness in a 
subspace of $C_b^1(\R)$ for a nonlinear evolution equation recently 
proposed by Andrew C. Fowler to describe the dynamics of dunes.  
\end{abstract}


\section{Introduction}

\subsection{General setting}
Dunes are land formations of sand which are subject to different 
forms and sizes based on their interaction with the wind or water or 
some other mobile medium.  In the case of dunes in the 
desert their shapes depend mainly on the amount of sand available and 
on the change of the direction of the wind with time (see 
Herrmann and Sauermann \cite{hs:hs}).  Some examples of dune patterns 
are longitudinal, transverse, star and Barchan dunes, however, there are more 
than 100 categories of dunes. Dunes also occur under rivers, for similar reasons, 
but their shapes are less exotic in this case, because the flow is mainly 
uni-directional.

An interesting topic is to try to understand if 
the shape of a dune is maintained when it moves.  With 
regard to Barchan dunes, for example, Herrmann and Sauermann \cite{hs:hs} 
have given some arguments against the hypothesis 
that Barchan dunes are solitary waves, mainly because  they  constantly 
lose sand at the two horns and tend to disappear 
if  not supplied with new sand.  Recently, Dur\'an, Schw\"ammle 
and Herrmann \cite{dsh:dsh} considered a minimal model for dunes consisting 
of three coupled equations of motion to study numerically the 
mechanisms of dune interactions for the case when a small 
Barchan dune collides with a bigger one; four different cases were 
observed, depending only on the relative sizes of the two dunes, namely, 
coalescence, breeding, budding, and solitary wave behavior.

In this paper, we are concerned with the following evolution equation 
proposed by Fowler (see \cite{f0:f0}, \cite{f1:f1} and \cite{f2:f2} 
for more details) to study nonlinear dune formation:
\begin{equation}\label{eq:fow0}
  \frac{\partial u}{\partial t} (x,t) + \frac{\partial}{\partial x} \Big{[} 
   \frac{u^2}{2}(x,t) - \frac{\partial u}{\partial x}(x,t) + \int_0^{+\infty} 
  \xi^{-1/3} \frac{\partial u}{\partial x} (x-\xi, t) d\xi \Big{]}=0,
\end{equation}
where $u=u(x,t)$ represents the dune amplitude, $x \in \R$, and $t \ge 0$.  The 
second and fourth terms of equation (\ref{eq:fow0}) correspond to 
the nonlinear and nonlocal terms respectively, while the third term is the 
dissipative term.\par
Let us give a brief description of the model derivation. For more details, we 
refer to Fowler \cite{f0:f0, f1:f1, f2:f2}, which we follow closely. The model 
stems from the Exner law, which is the conservation of mass for the sediment:
$$ \frac{\partial u}{\partial t} +  \frac{\partial q}{\partial x}  = 0,$$ where
the bedload transport $q = q(\tau)$ is assumed, in the case of dunes, to depend 
only on the   stress  $\tau$ exerted by the fluid on the erodible bed.  We  
assume a two-dimensional flow, where $x$ is the horizontal direction and 
the second direction is the upwards coordinate orthogonal to $x$. This  should account for transverse dunes, but obviously not  for other types of dunes.  The nonlocal 
term in equation (\ref{eq:fow0}) arises from a subtle modelling 
of the basal shear stress $\tau_b$.  Roughly speaking, the turbulent bottom 
shear stress is given by $\tau_b \approx f \rho v^2$, where $\rho$ is the fluid 
density, $f$ is a dimensionless friction coefficient and $v$ is the mean fluid 
velocity (vertically averaged).  By performing an asymptotic expansion with 
respect to the aspect ratio $\epsilon$ of the evolving  bedform, 
$\epsilon =\frac{\text{bed thickness}}{\text{fluid depth}} \ll 1,$ 
and a perturbation analysis of a basic Poiseuille flow (Orr-Sommerfeld equation), 
Fowler \cite{f0:f0, f1:f1, f2:f2} was able to obtain the following expression:
  $$\tau_b \approx f \rho v^2 \left\{ 1-u + \alpha \int_0^{+\infty} 
  \xi^{-1/3} \frac{\partial u}{\partial x} (x-\xi, t)  d\xi \right\},$$
where $\alpha$ is a positive constant proportional to $Re^{1/3}$, $Re$ being 
the Reynolds number.  Due to the bed slope $\frac{\partial u}{\partial x}$, there 
is an additional force generated by gravity $g$. Therefore,  the net stress causing 
motion is actually $\tau = \tau_b - (\rho_s - \rho) g D_s\frac{\partial u}{\partial x}$, 
where $\rho_s$ is the sediment density and $ D_s$ the mean diameter of a 
sediment particle.  As long as $u$ is small, the shallow water approximation 
applies to the velocity $v$ and,  for small Froude number, the (dimensionless) mean 
fluid velocity  can be approximated by $v \approx \frac{1}{1-u}$.  Thus, the mean 
fluid velocity and the bottom shear stress  depend on the motion of the dune 
profile $u$, and therefore there is a feedback between the dune profile and the 
motion of the fluid. In dimensionless variables, taking all physical constants 
equal to $1$, the  resulting net stress is then given by
$$
  \tau \approx  1 + u + u^2 + \int_0^{+\infty} \xi^{-1/3} 
  \frac{\partial u}{\partial x} (x-\xi, t)\,  d\xi  - \frac{\partial u}{\partial x}.
$$ 
Notice that the nonlinear nonlocal term 
$ 2 u \int_0^{+\infty} \xi^{-1/3} \frac{\partial u}{\partial x} (x-\xi, t) d\xi$ has 
been discarded.   By a Taylor expansion, up to order $2$, we get  
$q(\tau) \approx q(1) + q'(1)(\tau -1) + \frac{1}{2}q''(1)(\tau -1)^2$.
Now, considering a moving spatial coordinate, i.e. replacing $x$ by
the new variable $x-q'(1)t$,   plugging $q$ into the Exner equation,  after a 
suitable rescaling,  we obtain the canonical equation (\ref{eq:fow0}). 
\par
Some numerical computations have been performed by Fowler 
\cite{f1:f1,f2:f2}  and Alibaud, Azerad  and Is\`ebe \cite{aai:aai}.  Fowler 
mentions the fact that the numerical solution, computed with a pseudo-spectral 
method in a large domain, starting from random initial data, converges to a final 
state consisting of one travelling-wave.  Alibaud et al., using a finite difference 
scheme valid for a bounded  time interval,  starting from a compactly supported 
nonnegative initial data, showed that the  numerical solution of the Fowler 
equation (\ref{eq:fow0}) quickly evolves to a solution with  a non zero negative 
part, showing the erosive effect of the nonlocal term. They also establish 
theoretically the non monotone property of (\ref{eq:fow0}), namely the 
violation of the maximum principle  (see also Remark \ref{Remark:nonmonotone} below).
\par  
To the authors' knowledge, ours is the first study to report a 
rigorous mathematical proof of the existence of travelling-waves for dune 
morphodynamics.  We notice that we have not found nontrivial travelling-waves 
of the solitary-wave type for this model (see Remark \ref{R:Exist-soliton} below), 
however we could not exclude the possibility that they exist.  What we obtain is 
more bore-like travelling-waves.  This type of travelling dunes has not been 
observed yet, to the author's knowledge.  This may put under question the 
validity of the Fowler equation to faithfully describe dune morphodynamics.  The 
authors hope that these results could be of interest for geographers, geologists, 
oceanographers and others.

\subsection{Organization of the paper}
In Section \ref{section:travelling} we study the existence of travelling-wave 
solutions to equation (\ref{eq:fow0}).  The main result 
of this section is Theorem \ref{theo:exist} which implies that for each  
wave speed $d>0$, and $\eta$ in a neighborhood of zero, $\eta \in \R$, there 
exists a travelling-wave solution $u(x,t)=\phi(x-dt)$ to the following 
version of equation (\ref{eq:fow0}) 
\begin{equation*}
  \frac{\partial u}{\partial t} (x,t) + \frac{\partial}{\partial x} \Big{[} 
   \frac{u^2}{2}(x,t) - \frac{\partial u}{\partial x}(x,t) + \eta \int_0^{+\infty} 
  \xi^{-1/3} \frac{\partial u}{\partial x} (x-\xi, t) d\xi \Big{]}=0,
\end{equation*}
where $\phi \in C^1_b(\R)$; the idea of its proof is to use the implicit 
function theorem on suitable Banach spaces.  Then, by a scaling argument and 
considering a suitable translation of the travelling-wave, we extend this 
result for any $\eta \in \R$ and any wave speed $d\in \mathbb{R}$.

Section \ref{section:local} is devoted to proving local well-posedness (LWP) 
for the integral equation associated to the initial value problem (IVP) 
for equation (\ref{eq:fow0}).  Inspired by the regularity of the 
travelling-wave obtained in Section 2, we consider a suitable subspace of 
$C^1_b(\R)$.  The analysis of the linear equation associated to equation 
(\ref{eq:fow0}) is addressed in Sub-section \ref{subsection:linear}.   Next, 
in Sub-section \ref{subsection:local}, the main result of this section is 
stated in Theorem \ref{T:local}; it gives local-in-time existence of the solution of 
the integral equation associated to the IVP for equation (\ref{eq:fow0}), with initial 
data belonging to the subspace $X$  of $C^1_b(\R)$, where 
$$X :=  \{f \in C_b^1(\R) ; f' \text{ is uniformly continuous}  \}.$$

\subsection{Notations}

- We denote by $\R$ and $\C$ the sets of all real and complex 
  numbers respectively. $\N$ denotes the set of all natural numbers.\\
- We denote by $C(c_1, c_2, \ldots)$ a constant which depends on the 
  parameters $c_1, c_2, \ldots$  $C$ is assumed to be a non-decreasing 
  function of its arguments. \\
- The norm of a measurable function $f\in L^p(\Omega)$, for $\Omega$ a 
  subset of $\mathbb R$, is written
  $\|f\|^p_{L^p(\Omega)}=\int_{\Omega}|f|^pdx$ for $1\le p < +\infty$, and 
  $\|f\|_{L^\infty(\Omega)}=\text{ess sup}_{\Omega}|f|$. The inner
  product of two functions $f,g \in L^2(\Omega) $ is written as
  $(f,g)=\int_{\Omega}f\bar gdx$.  We will often omit set $\Omega$ when 
  context is clear.\\
- We denote by $\hat f = \F f$ the Fourier transform of $f$ ($\F ^{-1}$ and 
  $\; \check{} \;$  are used to denote the inverse of the Fourier transform), where 
  $\hat f (\xi) := \frac{1}{\sqrt{2\pi}} \int e^{-i\xi x}f(x)dx$  
  for $f \in L^1(\R)$ (it follows that $\widehat{f * g}= \sqrt{2\pi} 
  \hat f \hat g$ for $f, g \in L^1(\R)$). \\
- The Schwartz space of rapidly decreasing functions on $\R$ is denoted 
  $\mathcal S(\R)$. \\
- We denote $\Lambda:=(1-\partial^2_x)^{1/2}$ and $H^s(\R)$ ($s\in\R$)
  the usual Sobolev space 
  $H^s(\R)=\{u\in {\mathcal S}'(\R), \|u\|_{H^s}<\infty\}$, 
  where $\| u\|_{H^s}=\| \Lambda^s u\|_{L^2}$. \\
- Let $\Omega \subset \R$.  $C^0(\Omega)= C(\Omega)$ is used to denote the space 
  of all continuous complex-valued functions on $\Omega$.  Moreover, 
  $C^k(\Omega) = \{u:\Omega \mapsto \C \; ; \; u, u', \ldots, u^{(k)}
  \in C^0(\Omega)\}$, for $k \in \N$.  We write $C^{\infty} (\Omega)$ to denote 
  the set of infinitely differentiable complex-valued functions on $\Omega$.  Similarly, 
  we use the notations $C^0(\Omega;Y)= C(\Omega;Y), C^k(\Omega; Y), C^\infty (\Omega;Y)$ 
  when functions take values in the Banach space $Y$.\\
- We write $C_\infty(\R)$ to denote the space of all continuous complex-valued 
  functions defined on $\R$ which tend to zero at infinity.\\
- We denote by $C_b(\R)=C_b^0(\R)$ the space of all bounded continuous real-valued 
  functions on $\R$ with the norm $\| \cdot \|_{L^\infty}$.  Moreover, for every 
  $k \in \N$, we write 
   $$
    C_b^k(\R):= \{f\in C^k(\R) \; ; \; f, f', \ldots, f^{(k)} \in C_b(\R) \},
   $$
  where $\|f\|_{C_b^k} := \sum_{i=0}^k\| f^{(i)} \|_{L^\infty}$, for all 
  $f \in C_b^k(\R)$. \\
- If $X$ and $Y$ are two Banach spaces, we denote by ${\mathfrak L}(X,Y)$ 
  the set of all continuous linear mappings defined on $X$ 
  with values in $Y$; if $X=Y$, we denote by $\mathfrak{L}(X)$.

\section{Existence of  Travelling-Wave Solutions of the
Fowler Equation}\label{section:travelling}
We begin this section with some notations and preliminary results.
We define 
\begin{equation}\label{eq:psi}
 \psi (x):= \chi_{(0,\infty)}(x) \cdot x^{-1/3}, 
 \;\; \text{for all} \;\; x \in \R, 
\end{equation}
where $\chi_A$ is used to denote the characteristic 
function of the set $A$.  We also define 
\begin{equation}\label{eq:g[u]}
 g[u]:= \psi * \partial_x u.
\end{equation}   
We note that, since $\psi \in {\mathcal S}'(\R)$, 
it follows that for $\phi \in \mathcal S(\R)$, 
one has that $\psi * \phi \in C^\infty(\R)\cap{\mathcal S}'(\R)$ and  
$\widehat{\psi * \phi} = \sqrt{2\pi} \hat \psi \hat\phi$ (see Rudin \cite{r:r}).   
Then,  for $\varphi \in \mathcal{S}(\R)$, $g[\varphi] (\cdot)= 
\psi * \partial_x \varphi (\cdot)= \int_0^{+\infty} \xi^{-1/3} 
\partial_x \varphi (\cdot-\xi) d\xi$. Next lemma gives the 
Fourier transform of function $\psi$.
\begin{Lemma}\label{lema:foupsi} 
For the function $\psi$ defined by (\ref{eq:psi}) we have  
\begin{equation}\label{eq:fourier}
  \hat \psi (\xi) = \frac{1}{\sqrt{2 \pi}} \Gamma\Big(\frac23\Big) 
  \Big(\frac12 -i \frac{\sqrt 3}{2} \text{sgn}(\xi) \Big)
  |\xi|^{-2/3},
\end{equation}
where 
\begin{equation*}
\text{sgn}(\xi)  =\left\{
  \begin{array}
  [c]{r}
  -1,\text{ } \xi <0,\\
  1,\text{ } \xi>0,
  \end{array}
\right.
\end{equation*}
and $\Gamma$ is the gamma function.
\end{Lemma}
\begin{proof}
We define the function $\psi_n (x) := \chi_{(0,n)}(x) x^{-1/3}$, for all 
$x \in \R$, and $n \in \N$.  It is not difficult to see 
that $\psi_n \rightarrow \psi$ in ${\mathcal S}'(\R)$ as $n$ 
goes to infinity.  Let $\varphi \in {\mathcal S}(\R)$.   Then
\begin{eqnarray*}
  \langle \hat \psi_n, \varphi \rangle 
  &=& \frac{1}{\sqrt{2\pi}} \int \Big[ 
  \int_0^n \frac{\cos (\xi x)}{\xi^{1/3}} d\xi 
  -i \int_0^n \frac{\sin (\xi x)}{\xi^{1/3}} d\xi \Big] 
  \varphi (x) dx\\
  &=& \frac{1}{\sqrt{2\pi}} \int x^{-2/3} \Big[ 
  \int_0^{nx} \frac{\cos (u)}{u^{1/3}} du 
  -i \int_0^{nx} \frac{\sin (u)}{u^{1/3}} du \Big] 
  \varphi (x) dx.
\end{eqnarray*}
Since 
\begin{equation*}
  \int_0^{+\infty} \frac{\cos x}{x^{1/3}} dx = \frac12 \Gamma\Big(\frac23\Big), \;\;
  \text{and} \;\; 
  \int_0^{+\infty} \frac{\sin x}{x^{1/3}} dx = \frac{\sqrt 3}{2} 
  \Gamma\Big(\frac23\Big), 
\end{equation*}
it follows that 
\begin{equation*}
  \Big|x^{-2/3} \int_0^{nx} \frac{e^{-iu}}{u^{1/3}} du \; 
  \varphi(x) \Big| \le C |x|^{-2/3} |\varphi(x)|,
\end{equation*}
for all $n \in \N$, and $x \in \R$.  Therefore, the dominated convergence 
theorem implies that 
\begin{equation*}
 \lim_{n \rightarrow \infty} \langle \hat \psi_n, \varphi \rangle 
 = \frac{1}{\sqrt{2\pi}} \int x^{-2/3} \Gamma \Big( \frac23 \Big) 
   \Big(\frac12 - i \frac{\sqrt 3}{2} \text{sgn} (x)  \Big) \varphi(x) dx.
\end{equation*}
This completes the proof of the lemma.
\end{proof}
\begin{Remark}\label{R:def}
Let $s \in \R$. If $u \in H^s(\R)$, one can define $g[u]$  
through its Fourier transform by 
\begin{equation}\label{eq:fourierg}
\widehat{g[u]}(\xi) := \Gamma \Big( \frac23 \Big) \Big( 
\frac{\sqrt 3}{2} \text{sgn} (\xi)+\frac{i}{2} \Big) \xi^{1/3} \hat u (\xi),
\end{equation} 
for almost every $\xi \in \R$.  Thus, if $u \in H^s(\R)$, it follows 
that  $g[u] \in H^{s-1/3}(\R)$ and 
$ \| g[u] \|_{H^{s-1/3}} \le \Gamma \Big( \frac23 \Big) \|u\|_{H^s}$.  
\end{Remark}
In this section we consider the following, more general, version 
of equation (\ref{eq:fow0}):
\begin{equation}\label{eq:fow}
  \partial_t u (x,t)+ \partial_x \Big ( \frac{u^2}{2}  
  - \partial_x u + \eta \; g[u]\Big) (x,t)=0,
\end{equation}
where $\eta \in \R$.  We will show existence of travelling-wave 
solutions to equation (\ref{eq:fow}), for any $\eta \in \R$.  First, 
we consider the case $\eta=0$.  For any $d\in \R$ 
(see Johnson \cite{j:j}), 
\begin{equation}\label{eq:eta=0}
 u_d(x,t)= \frac{d}{2} \Big[ 1-\tanh \Big(\frac{d}{4} (x-\frac{d}{2}t) 
 \Big) \Big]
\end{equation}
is a solution to equation (\ref{eq:fow}) with $\eta=0$.
\begin{Remark}\label{R:scale}
Let $\lambda >0$.  We define
\begin{equation}\label{eq:scale} 
  u_\lambda (x,t) := \frac{1}{\lambda} u\Big(\frac{x}{\lambda}, 
  \frac{t}{\lambda^2} \Big), \;\;\text{for}  \;\; x \in \R, 
  \;\; \text{and}\;\; t \ge 0.  
\end{equation}
It is straightforward to check that if $u$ is a solution to the equation
\begin{equation}\label{eq:fow1}
  \partial_t u (x,t)+ \partial_x \Big ( \frac{u^2}{2}   
  - \partial_x u + \lambda^{2/3} \eta \; g[u] \Big)(x,t)=0,
 \end{equation}
then $u_\lambda$ satisfies equation (\ref{eq:fow}).  Hence, if 
$\phi$ is a travelling-wave solution of equation (\ref{eq:fow1}) 
with speed $c$, then $\phi_\lambda (\cdot)= \frac{1}{\lambda} 
\phi(\frac{1}{\lambda} \cdot)$ is a travelling-wave solution of 
equation (\ref{eq:fow}) with speed $c/\lambda$.
\end{Remark}
We define, for $c \in \R$, the functions 
\begin{equation}\label{eq:gc}
  g_c(x):= c \Big(1- \tanh \Big( \frac{c}{2} x \Big)  \Big), \;\;\;\text{and} \;\;\;
  h_c(x) := g_c'(x) = -\frac{c^2}{2} \text{sech}^2 \Big( \frac{c}{2} x \Big).
\end{equation}  
\begin{Remark}\label{R:gc} 
Let $c \in \R$.  We see that $g[g_c]=I_1+I_2$, 
where $I_j := \psi_j * h_c$ for $j=1,2$, with 
$\psi_1:= \psi \cdot \chi_{(0,1)}$, and  
$\psi_2 := \psi \cdot \chi_{(1,+\infty)}$.  Now, we state 
some immediate properties of the function $g[g_c]$.\\
{\bf{a.)}} Let $p>3$.  Since $\psi_1 \in L^1(\R)$, and $\psi_2\in L^p(\R)$, it 
follows from the Young inequality for convolution that $g[g_c] \in L^p(\R)$. \\
{\bf{b.)}} Furthermore,  $g[g_c] \in C_\infty(\R)$.  In fact, it follows from 
the dominated convergence theorem that $I_1$ is continuous and 
$I_1(x) \rightarrow 0$ as $|x| \rightarrow \infty$.  Moreover, H\"older's 
inequality and the dominated convergence theorem imply that
\begin{equation*}
 I_2(x) \le C(c) \Big( \int_1^{+\infty}  
 \frac{\text{sech}^4(\frac{c}{2} (x-\xi))}{\xi^{4/3}} d\xi \Big)^{1/4} 
 \rightarrow 0, \; \text{as} \; |x|\rightarrow +\infty.
\end{equation*}
The continuity of $I_2$ is shown similarly to the continuity of $I_1$.
\end{Remark}

Let $c\in \R$.  In the sequel, we will consider the following spaces:
\begin{eqnarray*}
  \X =\X_c &:=&  \Big\{ \varphi \in C_b^1(\R) \; \; ; \;\;  
  \int \varphi' h_c' \; dx=0\Big\},  \\ 
  \tilde \X = \tilde \X_c&:=& \Big\{ g_c+\varphi \; \; ; \;\; \varphi \in \X \Big\}, 
\end{eqnarray*}
where $\| \cdot \|_{\X} = \- \| \cdot \|_{C_b^1}$.  One sees that 
$(\X, \| \cdot \|_{\X})$ is a Banach space.
\begin{Remark}\label{R:property}
Assume that $\varphi \in C_b^1(\R)$.  By integration by parts one has that 
\begin{equation}\label{eq:new}
   g[\varphi](x) = \int_0^1 \frac{1}{\xi^{1/3}} \varphi'(x-\xi) d\xi
   +\varphi(x-1) - \frac13 
   \int_1^{+\infty} \frac{1}{\xi^{4/3}} \varphi(x-\xi) d\xi.
\end{equation}
Then $g[\varphi] \in C_b(\R)$.  Moreover,  
\begin{equation}\label{eq:ineqnew}
  \| g[\varphi] \|_{L^{\infty}} \le C   \; \| \varphi \|_{C_b^1}.
\end{equation}
Hence, if $\phi \in \tilde \X$, it follows from Remark \ref{R:gc}-{\bf{b.)}}  
above that $g[\phi] \in C_b(\R)$.
\end{Remark}

Suppose now that $u(x,t)=\phi(x-ct)$ is a solution to 
equation (\ref{eq:fow}), where $\phi \in \tilde \X$.  Then  
$-c \phi' + \frac{d}{dx}(\frac{\phi^2}{2}-\phi'+\eta g[\phi])=0$.  Thus,  
a sufficient condition to guarantee that $\phi$ satisfies the last 
equation is 
\begin{equation}\label{eq:travelling1}
   F(\eta,\phi)= F_c(\eta,\phi):= c\phi - \frac{\phi^2}{2}+\phi' -\eta g[\phi] =0.
\end{equation}
We denote by $\tau_c$ the function given by   
$\tau_c(x) := c-g_c(x)= c \; \text{tanh}(\frac{c}{2}x)$, for $x \in \R$.  
We now define the function $G=G_c$, which is well defined on $\R \times \X$ 
by Remarks \ref{R:gc} and \ref{R:property} above, as
\begin{eqnarray}\label{eq:travelling2}
   G:\R \times \X &\mapsto& C_b(\R) \nonumber \\
   (\eta,\varphi)    &\mapsto&
   G(\eta,\varphi) = \tau_c \varphi - \frac{\varphi^2}{2} +\varphi'
   -\eta g[\varphi] -\eta g[g_c].
\end{eqnarray}
Assume that $\phi= \varphi+g_c \in \tilde \X$.  Since $F(0,g_c)=0$, it 
follows that $F(\eta, \phi) = G(\eta,\varphi)$.  Hence, $\phi$ satisfies 
equation  (\ref{eq:travelling1}) if and only if $\varphi$ verifies the equation 
$G(\eta,\varphi)=0$.

The following theorem  implies the existence of a travelling-wave 
solution, $u(x,t)=\phi(x-ct)$ with $c>0$ and $\phi \in \tilde \X$, to 
equation (\ref{eq:fow}) for $\eta$ in a neighborhood of zero; its proof 
uses the implicit function theorem. 
\begin{Theorem}\label{theo:exist}
Suppose $c >0$.  Then there exist $\delta, \delta_0>0$ such that 
for every $\eta \in (-\delta,\delta)$, there is exactly one 
$\varphi_\eta=\varphi_{\eta,c} \in \X$ for which 
$\|\varphi_\eta\|_{\X} \le \delta_0$ and $G(\eta,\varphi_\eta)=0$.  Moreover, 
the mapping $\eta \mapsto \varphi_\eta$ is a $C^{\infty}$-map on a 
neighborhood of $0$. 
\end{Theorem}
\begin{proof}
Let $c>0$.  The mapping $G=G_c$ is defined on the Banach space 
$\R \times \X$ taking values in the Banach space 
$(C_b(\R), \| \cdot \|_{L^\infty})$, and satisfies $G(0,0)=0$.  

We now claim that $\partial_1 G$ and $\partial_2 G$ exist as 
partial F-derivatives (Fr\'echet derivative) on $\R \times \X$ and that 
the partial F-derivative $\partial_2 G(0,0):\X \mapsto C_b(\R)$ is bijective.  \\
In fact, let us take $(\eta, \varphi) \in \R \times \X$.   One can see 
that 
\begin{equation*}
  \partial_1 G(\eta,\varphi) \cdot 
   =-(g[\varphi] + g[g_c]) \cdot 
\end{equation*}
and 
\begin{equation}\label{eq:partial_2G}
  \partial_2 G(\eta,\varphi) \cdot 
  =(\tau_c -\varphi) \cdot +\partial_x \cdot 
  -\eta g[\cdot]. 
\end{equation}
Then $\partial_1 G(\eta,\varphi) \in \mathfrak L(\R, C_b(\R))$, and 
$\partial_2G(\eta,\varphi) \in \mathfrak L(\X,C_b(\R))$.  Moreover, 
we obtain that 
$\|\partial_1 G(\eta,\varphi)\|_{\mathfrak L(\R, C_b(\R))} \le C \cdot
(\|\varphi \|_{C^1_b} + \|g[g_c]\|_{L^\infty})$, and  
$\|\partial_2G(\eta,\varphi) \|_{\mathfrak{L}(\X,C_b(\R))} \le C \cdot 
(1+|\eta|+\| \tau_c - \varphi\|_{L^\infty})$, where we have 
used inequality (\ref{eq:ineqnew}).  Hence, $\partial_1G, \partial_2G$ 
exist as partial F-derivatives on $\R\times \X$. \\
We will now show that the partial F-derivative 
$ \partial_2G(0,0) =\tau_c+\partial_x:\X \mapsto C_b(\R)$ is bijective. 
We begin with the injectivity; we emphasize here that the definition of the 
space $\X \subset C^1_b(\R)$ was chosen to ensure the injectivity of 
the mapping $ \partial_2G(0,0)$.  Let $f$ be an element of $\X$ such that 
$\tau_c f +f'=0$.  By solving the last ordinary differential equation, one gets
\begin{equation*}
  f(x)=f(0) \; \cdot \; e^{-\int_0^x \tau_c(s)ds} 
  = f(0) \; \cdot \; \text{sech}^2\Big(\frac{c}{2}x\Big).
\end{equation*}
Since $f\in \X$, it follows that 
\begin{equation*}
  \int f'(x) h_c'(x) dx = -f(0)\frac{2}{c^2} \int (h_c')^2(x) dx =0.
\end{equation*}
Then $f(0)=0$, and therefore $f=0$. \\ 
We will now show that the mapping $\partial_2G(0,0)$ is onto.  Let $y$ be 
an element of $C_b(\R)$.  By the method of variation of parameters, we 
obtain that the function 
\begin{equation}\label{eq:sobrej}
  g(x) := \lambda l_c (x) + l_c (x)
  \int_0^x \frac{y(s)}{l_c(s)} ds
\end{equation}
is a solution to the equation $\tau_c g +g' = y$, for any  
$\lambda \in \R$, where $l_c:= -\frac{2}{c^2}h_c = 
\text{sech}^2(\frac{c}{2}x)$.  We will 
prove that $g \in \X$ for a suitably chosen real number $\lambda$.  
First, we remark that there exists a unique 
$\lambda=\lambda_{y,c} \in \R$ such that $\int g'h_c'dx$=0.  In  
fact take  
\begin{equation}\label{eq:lambda}
  \lambda := \frac{c^2}{2\int(h_c')^2(x)dx}
  \int \Big[ (h_c')^2(x) \int_0^x \frac{y(s)}{h_c(s)}ds
  +y(x) h_c'(x) \Big]dx,
\end{equation}
where we note that 
\begin{equation*}
  0< \int (h_c')^2(x)dx \le \frac{c^6}{4} \int
  \text{sech}^4 \Big(\frac{c}{2}x\Big) dx = C(c),  \;\;\;
  \int|h_c'(x)|dx =c^2, 
\end{equation*}
and
\begin{eqnarray*}
  && \int (h_c')^2(x) \Big| \int_0^x \frac{y(s)}{h_c(s)}ds \Big| dx 
     \le \frac{c^4}{2} \|y\|_{L^\infty} 
     \int \frac{\sinh^2(\frac{c}{2}x)}{\cosh^6(\frac{c}{2}x)}
     \Big| \int_0^x \frac{1+\cosh(cs)}{2} ds \Big| dx \\
  &&\le \frac{c^3}{4} \|y\|_{L^\infty} 
    \int \Big[ c|x| \text{sech}^4 \Big( \frac{c}{2}x \Big)  
    + 2 \; \text{sech}^2 \Big( \frac{c}{2}x \Big) \Big] dx
    \le C(c) \|y\|_{L^\infty}.
\end{eqnarray*}
It remains to show that $g$ given by (\ref{eq:sobrej}) and 
(\ref{eq:lambda}) belongs to $C_b^1(\R)$.  It is immediate to see that 
$g \in C(\R)$, we need to show that $g$ is bounded.  We have that 
\begin{eqnarray*}
   && \text{sech}^2 \Big( \frac{c}{2}x \Big) \Big| \int_0^x 
      \frac{y(s)}{\text{sech}^2(\frac{c}{2}s)} ds \Big| 
      \le \frac{\|y\|_{L^\infty}}{1+\cosh(cx)} 
      \Big|\int_0^x (1+\cosh(cs)) ds  \Big| \\
   && \le \|y\|_{L^\infty} \Big( 
      \frac{|x|}{1+\cosh(cx)} +\frac{1}{c} |\tanh(cx)|\Big)
      \le C(c) \|y\|_{L^\infty}.
\end{eqnarray*}
Then $g \in C_b(\R)$.  Moreover, since $g$ satisfies the 
equation $\tau_c g +g'=y$, it follows that $g \in C_b^1(\R)$.  Hence,   
$g\in \X$.  Therefore, $\partial_2 G(0,0)$ is a surjective mapping. 

It is not difficult to see, by using inequality (\ref{eq:ineqnew}), 
that $G$, $\partial_1 G$ and $\partial_2 G$ are continuous 
on $\R \times \X$.  Then, the implicit function theorem   
implies the first part of the theorem.  Furthermore, from (\ref{eq:travelling2}) 
one can see that function $G$ is quadratic in $\varphi$ and linear 
in $\eta$, therefore it is not difficult to verify that  
$\partial_{i,j}^2G(\eta,\varphi)$ is independent of 
$(\eta,\varphi) \in \R \times \X$, for all $i,j\in \{1,2\}$.  Hence,   
$\partial_{i_1,\ldots,i_k}^kG(\eta,\varphi)=0$ for all $k\ge3$, where 
$i_1,\ldots,i_k\in\{1,2\}$, and  $(\eta,\varphi) \in \R \times \X$.
Finally, the second part of the theorem is then a consequence of the fact 
that the mapping $G$ is a $C^\infty$-map on $\R \times \X$. 

\end{proof}
\begin{Corollary}
Let $\eta \in \R$ and $d\in \mathbb{R}$.  Then there is a travelling-wave 
solution $\tilde \phi  \in C^1_b(\R)$ of equation (\ref{eq:fow}) with speed $d$. 
\end{Corollary}
\begin{proof}
Let $c>0$.  By Theorem \ref{theo:exist} there exists 
$\lambda_0=\lambda_0(\eta,c)>0$ such that 
for every $\lambda \in (0,\lambda_0)$, there is a 
$\phi=\phi_{\lambda,\eta,c} \in C_b^1(\R)$ such that $u(x,t)=\phi(x-ct)$ is 
a solution to equation (\ref{eq:fow1}).  Now we can see, by using 
Remark \ref{R:scale},  that 
$\phi^{\dagger}(\cdot)=\frac{1}{\lambda}\phi (\frac{1}{\lambda}\cdot)$  
is a travelling-wave solution of equation (\ref{eq:fow}) with speed 
$c/\lambda \in (\frac{c}{\lambda_0}, +\infty)$.  The result now follows 
from the fact that if $\phi^{\dagger}(x-\tilde{c}t)$ is a solution of 
equation (\ref{eq:fow}) for some $\tilde c >0$, then 
$\tilde \phi(x,t):= \phi^{\dagger}(x-(\tilde{c}+k)t)+k$ is also a solution 
for all $k \in \mathbb{R}$.
\end{proof}

\begin{Remark}\label{R:Exist-soliton}Let us comment about the existence of solitary travelling waves.
Let us proceed formally at first.  By multiplying the equation 
$-c \phi' + \frac{d}{dx}(\frac{\phi^2}{2}-\phi'+\eta g[\phi])=0$
by  $\phi$, then integrating between $-\infty$ and $x$, assuming that 
$\phi, \phi' \rightarrow 0$ as $|x|\rightarrow +\infty$, we get 
$$
  -c \frac{\phi^2(x)}{2}+\frac{\phi^3(x)}{3}
  -\int_{-\infty}^{x} \phi(y) \phi''(y)dy 
  +\eta \int_{-\infty}^x \frac{dg[\phi]}{dy}(y) \phi(y)dy=0.
$$
Making $x\rightarrow +\infty$, integrating by parts, and then 
applying Parseval's relation and  Remark \ref{R:def}, we obtain
\begin{equation}\label{eq:soliton}
  \int_{-\infty}^{+\infty} \Big(\xi^2-\frac{\eta}{2} 
  \Gamma\Big(\frac23\Big) \xi^{4/3} \Big) 
  |\hat \phi (\xi)|^2 d\xi =0.
\end{equation}
These formal steps can be justified by assuming for instance that 
$\phi \in H^2(\mathbb{R})$.  Thus, equation (\ref{eq:soliton}) implies 
that if $\eta \le 0$, then $\phi=0$.  We can then conclude that there 
are no nontrivial travelling-waves of the solitary-wave type 
for equation (\ref{eq:fow}) when $\eta \le 0$.  However, in the 
physical case, that is to say when $\eta=1$ or more generally when 
$\eta>0$, equation (\ref{eq:soliton}) does not preclude the possibility that 
they may exist.
\end{Remark}

\section{Local Theory in a subspace of $C^1_b(\R)$}\label{section:local}
In Section \ref{section:travelling}, we proved the existence of a 
travelling-wave solution $u(x,t)=\phi(x-ct)$ to equation (\ref{eq:fow}) for 
any $\eta \in \R$, where $c$ is an appropriate positive number and $\phi\in C^1_b(\R)$. 
Motivated by this last result, we will consider in this section the local 
well-posedness theory for the following initial value problem (IVP) 
\begin{equation}\label{eq:IVP}
 \left \{
   \begin{array}{l}
     \partial_t u (x,t)+ \partial_x \big ( \frac12 u^2   
     -\partial_x u +  g[u] \big)(x,t)=0, \\
     u(0)=u_0,
   \end{array}
  \right.
\end{equation}
where $g[u]$ is given by (\ref{eq:g[u]}), and $u_0$ belongs to a 
suitable subspace of $C^1_b(\R)$.  The Cauchy problem associated to the IVP 
(\ref{eq:IVP}) for initial data $u_0 \in L^2(\R)$ was recently studied by 
Alibaud, Azerad and Is\`ebe \cite{aai:aai}.


\subsection{The Linear Equation}\label{subsection:linear}
First, we consider the linear part associated to the IVP (\ref{eq:IVP}), namely
\begin{equation}\label{eq:IVP-linear}
 \left \{
   \begin{array}{l}
     \partial_t u (x,t)  
     -\partial_x^2 u (x,t)+ \partial_x g[u](x,t) =0, \\
     u(0)=u_0.
   \end{array}
  \right.
\end{equation}
By formally taking the Fourier transform of the last expression, we get  
\begin{equation}\label{eq:sol-linear}
  \hat u(\xi,t) = \hat K(\xi,t) \hat u_0(\xi),
\end{equation}   
where
\begin{equation}\label{eq:K}
 \hat K(\xi,t) = e^{-t[\xi^2-\xi^{4/3}(a+ib \; \text{sgn}(\xi))]},  
\end{equation}
for $\xi \in \R$ and $t\ge 0$, with $a:=\frac12 \Gamma(\frac23)$ and 
$b:=-\frac{\sqrt 3}{2} \Gamma(\frac23)$.  For $\xi \in \R$, we define 
\begin{equation}\label{eq:Phi}
  \Phi(\xi) := (a+ib \; \text{sgn}(\xi)).
\end{equation}
We note that $|\Phi(\xi)|= \Gamma(\frac23)$, for all $\xi \in \R$.
\begin{Remark} The non local term $\partial_x g[u]$ is anti-dissipative of order $4/3$.
\end{Remark}
  
\begin{Remark}\label{Remark:nonmonotone}
For every $t>0$, the kernel $K(\cdot,t)$ is not a nonnegative function. Indeed, by contradiction, 
if $K(\cdot,t)$ would be nonnegative, one could bound 
$$|\hat{K}(\xi,t)| \leq \big{|} \frac{1}{\sqrt{2\pi}}\int e^{-i\xi x} K(x,t) dx  \big{|} \leq \hat{K}(0,t) = 1.$$
But, on the other hand,
$|\hat{K}(\xi,t)| = e^{-t[\xi^2-a \xi^{4/3}]} >1,\; \mbox{for} \; 0<|\xi| < a^{3/2}.$
Hence, for every $t>0$, there exists $x\in \R$ such that 
$K(x,t)<0$.  This fact implies, in particular, that the IVP 
(\ref{eq:IVP}) is non-monotone (see \cite{aai:aai} for more details).
\end{Remark}
For $t\ge 0$, we define the operator $E(t)$ by 
\begin{equation}\label{eq:ope}
 \left \{
   \begin{array}{l}
     E(t) \phi (x) = \frac{1}{\sqrt{2\pi}}\big(K(\cdot,t)*\phi  \big) (x), 
     \;\;\;\text{for} \;\; t>0 \;\; \text{and} \;\; x \in \R, \\ 
     E(0)\phi=\phi,
   \end{array}
  \right.
\end{equation}
where $\phi \in C_b(\R)$ (see Lemma \ref{Lemma:semig} below).  Now, we define 
the following spaces 
\begin{eqnarray}\label{eq:space}
  Y &:=&  \{g \in C_b(\R) ; g \text{ is uniformly continuous}  \} \;; \\
  X &:=&  \{f \in C_b^1(\R) ; f' \text{ is uniformly continuous}  \}.
\end{eqnarray}
One can see that  $(Y, \|\cdot\|_{C_b(\R)})$, and 
$(X, \|\cdot\|_{C_b^1(\R)})$ are Banach spaces and that 
$X \hookrightarrow Y$.  In Sub-section \ref{subsection:local} 
we will show local-in-time well-posedness of the IVP (\ref{eq:IVP}), with 
initial data $u_0 \in X$.

The following lemma contains a calculus result.
\begin{Lemma}\label{Lemma:calculus}
Let $h:\R \mapsto \C$ be a function which satisfies the following conditions:\\
{\bf{i.)}} $h \in L^1(\R) \cap C_{\infty}(\R) \cap C^2(\R \setminus \{ 0 \})$; \\
{\bf{ii.)}} $h' \in L^1(\R)$, $|h'(x)| \rightarrow 0$ as $|x|\rightarrow +\infty$.  
Moreover, there exist $\lim_{x \downarrow0} h'(x)=h'(0^+)$, and 
$\lim_{x \uparrow0} h'(x)=h'(0^-)$; \\
{\bf{iii.)}} $h'' \in L^1(\R)$. \\
Then  $\hat h \in L^1(\R) \cap C_{\infty}(\R)$, and 
\begin{equation}\label{eq:calculus}
   \|\hat h \|_{L^1} \le \sqrt{\frac{2}{\pi}} \Big[ 
   \|h\|_{L^1} + |h'(0^+)-h'(0^-)| + \|h''\|_{L^1} \Big].
\end{equation}
\end{Lemma}
\begin{proof}
Since $h\in L^1(\R)$, it follows from the Riemann-Lebesgue lemma that 
$\hat h \in C_{\infty}(\R)$.  After using integration by parts twice, we see that  
\begin{equation*}\label{eq:calculus1}
  \hat h(\xi) = \frac{1}{\sqrt{2\pi}} \Big[ 
  \int_{-\infty}^{+\infty} h''(x) \frac{e^{-i\xi x}}{(i \xi)^2} dx 
  +\frac{h'(0^+)-h'(0^-)}{(i\xi^2)} \Big], \;\;\; \text{for } \; 
  \xi \not = 0.
\end{equation*}
Expression (\ref{eq:calculus}) follows from the last equation and from the 
fact that $\|\hat h\|_{L^{\infty}} \le \frac{1}{\sqrt{2\pi}} \|h\|_{L^1}$. 
\end{proof}
\begin{Remark}\label{Remark:calculus}
It is well-known that $W^{1,1}(\R) \subset C_{\infty}(\R) \cap AC(\R)$, 
where $AC(\R)$ denotes the space of all complex-valued functions, which are 
absolutely continuous on $\R$.  Therefore, it follows 
from Lemma \ref{Lemma:calculus} above that if 
$f \in W^{2,1}(\R)$, then $\hat f \in L^1(\R) \cap C_\infty(\R)$ and
\begin{equation}\label{eq:calculus2}
   \|\hat f \|_{L^1} \le \sqrt{\frac{2}{\pi}} \Big[ 
   \|f\|_{L^1} +  \|f''\|_{L^1} \Big].
\end{equation}
\end{Remark}
\begin{Remark}\label{Remark:Kernel}
Suppose now that $t \in (0,1)$.  Since
\begin{eqnarray*}
  K(x,t) &=& \frac{1}{\sqrt{2\pi}} \int e^{ix\xi} 
             e^{-t[\xi^2-\xi^{4/3}\Phi(\xi)]} d\xi \\
         &=& \frac{t^{-1/2}}{\sqrt{2\pi}} \int e^{i (t^{-1/2}x) \xi} \;
             e^{-[\xi^2-\xi^{4/3}\Phi(\xi)]} \;
             e^{-(1-t^{1/3})\xi^{4/3}\Phi(\xi)} d\xi,
\end{eqnarray*}
it follows that 
\begin{equation}\label{eq:Kernel}
    K(x,t) = t^{-1/2} \big(K(\cdot,1) * G(\cdot,1-t^{1/3})\big)(t^{-1/2}x),  \;\; 
    \text{for} \;\; x \in \R,
\end{equation}
where
\begin{equation}\label{eq:Kernel-G}
   G(\cdot,1-t^{1/3}) = \frac{1}{\sqrt{2\pi}} \F^{-1} 
   (e^{-(1-t^{1/3})\xi^{4/3}(a+ib \; \text{sgn}(\xi))}) (\cdot).
\end{equation}
\end{Remark}
The next three lemmas are elementary calculus results  which will be used 
in the sequel.
\begin{Lemma}\label{Lemma:calculus1}
Suppose that $\alpha>-1$, and $\beta>0$.  Then 
\begin{equation*}
  I(\alpha,\beta) :=\int |\xi|^\alpha e^{-\beta |\xi|^{4/3}} d\xi 
  = C(\alpha) \beta^{-\frac34(\alpha+1)}.
\end{equation*}
\end{Lemma}
\begin{proof}
The assertion of the lemma follows from the fact that 
  $$I(\alpha,\beta) =\beta^{-\frac34(\alpha+1)} 
  \int |\tau|^\alpha e^{-|\tau|^{4/3}} d\tau. $$ 
\end{proof}

\begin{Lemma}\label{Lemma:calculus2}
Suppose that $\alpha>-1$, $\beta>0$, and $t>0$.  Then 
\begin{equation*}
  I(\alpha,\beta,t) :=\int |\xi|^\alpha e^{-t[\xi^2-\beta |\xi|^{4/3}]} d\xi 
  \le C(\alpha,\beta) \Big[ e^{\frac{4}{27}\beta^3t} +t^{-\frac{\alpha+1}{2}}  \Big].
\end{equation*}
\end{Lemma}
\begin{proof}
It is elementary to check that $\xi^2-\beta\xi^{4/3} \ge -\frac{4}{27}\beta^3$ for all 
$\xi \in \R$, and $\xi^2-\beta\xi^{4/3} \ge \xi^2/2$ for $\xi\ge(2\beta)^{3/2}$.  Then
\begin{eqnarray*}
  I(\alpha,\beta,t) &\le& 2 \Big[ 
  \int_0^{(2\beta)^{3/2}} \xi^\alpha e^{\frac{4}{27}\beta^3t}  d\xi 
  +\int_{(2\beta)^{3/2}}^{+\infty} \xi^\alpha e^{-\frac{t}{2}\xi^2} d\xi\Big] \\ 
  &\le& C(\alpha,\beta) \Big[ e^{\frac{4}{27}\beta^3t} + 
  \int_0^{+\infty} \Big(\frac{2}{t}\Big)^{\frac{\alpha}{2}} u^\alpha 
  e^{-u^2} \sqrt{\frac{2}{t}} du\Big],
\end{eqnarray*}
where in the last inequality we have used the fact that $\alpha >-1$.  The 
result now follows.
\end{proof}
\begin{Lemma}\label{Lemma:calculus3}
Suppose that $g \in W^{1,1}(\R)$, and $l \in L^\infty (\R)$.  If  
$f:=g*l$, then  $f \in C^1(\R) \cap W^{1,\infty}(\R)$, and  
$f'(x)= (g'*l)(x)$ for all $x \in  \R$.
\end{Lemma}
\begin{proof}
Since $g \in L^1(\R)$ and $l \in L^\infty(\R)$, it follows from Young's inequality 
that $f \in L^\infty (\R)$.  Moreover, since 
$|f(x+h)-f(x)|\le \|g(\cdot+h)- g(\cdot) \|_{L^1} \|l\|_{L^\infty}$ for 
all $x\in \R$, and $ \|g(\cdot+h)- g(\cdot) \|_{L^1} \rightarrow 0$ as 
$h$ tends to zero, it follows that $f \in C(\R)$.  Let $x\in \R$.  Since  
$W^{1,1} (\R) \subset AC(\R)$, we see that 
\begin{eqnarray*}
  && \Big|\frac{f(x+h)-f(x)}{h} - \int g'(x-y) l(y) dy \Big| = \\
  && \Big|\int \int_0^1 \big(g'(x-y+th) - g'(x-y)\big) l(y) dtdy \Big| \\
  && \le \|l\|_{L^\infty}  \int_0^1 \|g'(\cdot+th) -g'(\cdot) \|_{L^1} dt \rightarrow 0 
  \;\;\; \text{as} \;\; h\rightarrow 0,
\end{eqnarray*}
where the last expression is a consequence of the dominated convergence theorem.
\end{proof}

The following five lemmas provide more explicit estimates than the corresponding 
results mentioned in \cite{aai:aai}.  The next lemma gives an upper bound, which 
goes to infinity as $t$ tends to $1$, for
$\| G(\cdot, 1-t^{1/3}) \|_{L^1}$ when $t \in [0,1)$.
\begin{Lemma}\label{Lemma:Kernel-G}
Let $t_0 \in(0,1)$.  Then, for all $t\in[0,t_0]$,  the function 
$G(\cdot,1-t^{1/3})$ given by (\ref{eq:Kernel-G}) belongs to 
$ L^1(\R) \cap C_{\infty}(\R)$.  Moreover, 
\begin{equation}\label{eq:Kernel-G-1}
  \| G(\cdot, 1-t^{1/3}) \|_{L^1} 
  \le C \big[ (1-t^{1/3})^{3/4} + (1-t^{1/3})^{-3/4}\big]
  \le C \cdot (1-t_0^{1/3})^{-3/4},
\end{equation}
for all $t\in[0,t_0]$, where $C$ is a positive constant independent of $t$.
\end{Lemma}
\begin{proof}
Let $t\in[0,1)$.  We define 
$g(\xi,t):=e^{-(1-t^{1/3})\xi^{4/3}\Phi(\xi)}$ for $\xi \in \R$.  It 
follows that  $g(\cdot,t)$ is continuous.  Furthermore,
\begin{equation}\label{eq:Kernel-G-1a}
  \partial_{\xi} g(\xi,t) = -\frac{4}{3} (1-t^{1/3}) \xi^{1/3} \Phi(\xi)
  e^{-(1-t^{1/3})\xi^{4/3}\Phi(\xi)},  \;\; \text{for} \; 
  \xi \not= 0. 
\end{equation}
Then $|\partial_\xi g (\xi,t)| \rightarrow 0$ as $|\xi| \rightarrow + \infty$, 
and $\partial_{\xi}g(0^+,t)=0=\partial_{\xi}g(0^-,t)$.  Moreover, 
\begin{eqnarray*}
  \partial_{\xi}^2 g(\xi,t) &=& \Big[-\frac{4}{9} (1-t^{1/3}) \xi^{-2/3} \Phi(\xi)
  +\Big(\frac43 (1-t^{1/3}) \xi^{1/3} \Phi(\xi) \Big)^2 \Big] \\
  && \times  
  e^{-(1-t^{1/3})\xi^{4/3}\Phi(\xi)},  \;\;\;\; \text{for} \;\; 
  \xi \not= 0.
\end{eqnarray*}
We see that $g(\cdot,t) \in C_{\infty}(\R) \cap C^2(\R \setminus \{ 0 \})$.  
In addition, $\|g(\cdot,t)\|_{L^1} = C \cdot (1-t^{1/3})^{-3/4}$, 
and $\|\partial_{\xi} g(\cdot,t)  \|_{L^1} =4$.  Furthermore,
\begin{eqnarray*}
  \|\partial_{\xi}^2 g(\cdot,t)\|_{L^1} &\le& 
  C \cdot (1-t^{1/3}) \int |\xi|^{-2/3} e^{-(1-t^{1/3}) a \xi^{4/3}} d\xi  \\
  && + C \cdot (1-t^{1/3})^2 \int |\xi|^{2/3} e^{-(1-t^{1/3}) a \xi^{4/3}} d\xi \\
  &\le& C \cdot (1-t^{1/3})^{3/4},
\end{eqnarray*}
where the last inequality is a consequence of Lemma \ref{Lemma:calculus1} above.  The 
result now follows from Lemma \ref{Lemma:calculus}.
\end{proof}


Lemmas \ref{Lemma:Kernel-K} and \ref{Lemma:Kernel-Kd} below provide estimates 
for $\| K(\cdot, t) \|_{L^1}$ and $\|\partial_x K(\cdot, t) \|_{L^1}$, for any $t>0$.  

\begin{Lemma}\label{Lemma:Kernel-K}
Suppose that $t>0$.  Then the function 
$K(\cdot,t) \in L^1(\R) \cap C_{\infty}(\R)$, and 
\begin{equation}\label{eq:Kernel-K-1}
  \| K(\cdot, t) \|_{L^1} 
  \le C \cdot \big( 1 + t^2 e^{\frac{4}{27}a^3t} \big), 
\end{equation}
where $C$ is a positive constant independent of $t$.
\end{Lemma}
\begin{proof}
Let $t>0$.  It follows from (\ref{eq:K}) that
\begin{equation*}
  \partial_\xi \hat K(\xi,t) = -t\Big[2\xi-\frac43 \xi^{1/3} 
  \Phi(\xi) \Big] \hat K(\xi,t), \;\;\;\; \text{for} \; \xi \not= 0, 
\end{equation*}
and 
\begin{equation*}
  \partial_\xi^2 \hat K(\xi,t) = \Big\{ -t\Big[ 2-\frac49 \xi^{-2/3} \Phi(\xi)
  \Big] +t^2 \Big[ 2\xi-\frac43 \xi^{1/3} 
  \Phi(\xi) \Big]^2  \Big\} \hat K(\xi,t),  \;\; \text{for} \; 
  \xi \not= 0.
\end{equation*}
Then $\hat K(\cdot,t) \in C_{\infty}(\R) \cap C^2(\R \setminus \{ 0 \})$.  Moreover, 
$|\partial_\xi \hat K (\xi,t)| \rightarrow 0$ as $|\xi| \rightarrow + \infty$, and 
$\partial_\xi \hat K(0^+,t)=0=\partial_\xi \hat K(0^-,t)$.  Furthermore,
\begin{equation*}
  \|\hat K (\cdot,t)\|_{L^1} = 2\int_0^{+\infty} e^{-t[\xi^2-a\xi^{4/3}]} d\xi 
  \le C\Big[ e^{\frac{4}{27}a^3t} +\frac{1}{\sqrt t} \Big],
\end{equation*}
where the last inequality is a consequence of Lemma \ref{Lemma:calculus2}.  
Again using Lemma \ref{Lemma:calculus2}, we see that 
\begin{eqnarray*}
  \|\partial_\xi \hat K (\cdot,t)\|_{L^1} &\le& C\big[ 1+t^{1/3} 
  +t \; e^{\frac{4}{27}a^3t}\big], \;\;\; \text{and} \\
  \|\partial_\xi^2 \hat K (\cdot,t)\|_{L^1} &\le& C\big[ \sqrt t + t^{5/6} 
  + t^{7/6} + (t+t^2) e^{\frac{4}{27}a^3t}\big] 
  \le C\big[\sqrt t + t^2 e^{\frac{4}{27}a^3t} \big].
\end{eqnarray*}
Lemma \ref{Lemma:calculus}, applied to $h(\cdot)= \hat K(\cdot,t)$, implies that 
\begin{equation}\label{eq:Kernel-K-1a}
  \| K(\cdot, t) \|_{L^1} 
  \le C \Big[ \frac{1}{\sqrt t} + t^2 e^{\frac{4}{27}a^3t} \Big], \;\;  
  \text{ for all } t>0.
\end{equation}
Suppose now that $t\in(0,1)$. Then   
\begin{eqnarray}\label{eq:Kernel-K-1b}
 \int |K(x,t)|dx 
 &=& \frac{1}{\sqrt t}
     \int \big| K(\cdot,1) * G(\cdot,1-t^{1/3})\big|(x/\sqrt t) dx \nonumber \\
 &=& \int (1+y^2)^{1/2} \frac{\big| K(\cdot,1) * G(\cdot,1-t^{1/3})
     \big|(y)}{(1+y^2)^{1/2}} dy 
     \nonumber \\
 &\le& C \; \big{\|} \hat K(\cdot,1) \hat G(\cdot, 1-t^{1/3})\big{\|}_{H^1},
\end{eqnarray}
where the first equality above comes from (\ref{eq:Kernel}). From  
the fact that $e^{-(1-h^{1/3}) a \xi^{4/3} } \le 1$, for all $h\in[0,1)$, 
$\xi \in \R$, and equation (\ref{eq:Kernel-G-1a}) we have that 
$|\hat G(\xi,1-h^{1/3})| \le C$, and 
$|\partial_\xi \hat G(\xi,1-h^{1/3})| \le C |\xi|^{1/3}$ for all 
$\xi \in \R$.  Now, from Lemma \ref{Lemma:calculus2}, we obtain 
\begin{eqnarray}\label{eq:unif1}
 &&\!\!\!\!\!\!\!\!\!\!\!\!\!\!\!\!\!\!\!\!
   \|\hat K(\cdot,1) \hat G (\cdot,1-h^{1/3})\|_{H^1} 
   \le \|\hat K(\cdot,1) \hat G (\cdot,1-h^{1/3})\|_{L^2} \nonumber \\
 &&+ \|\partial_\xi \hat K(\cdot,1) \hat G (\cdot,1-h^{1/3})\|_{L^2} 
   + \|\hat K(\cdot,1) \partial_\xi \hat G (\cdot,1-h^{1/3})\|_{L^2} \le C, 
\end{eqnarray}
for all $h\in [0,1)$.  The assertion of the lemma now follows from 
(\ref{eq:Kernel-K-1a})-(\ref{eq:unif1}).
\end{proof}
The following result gives an upper bound for $\|\partial_x K(\cdot,t) \|_{L^1}$  
when $t \in (0,1)$.
\begin{Lemma}\label{Lemma:Kernel-Kd-0} 
Let $t\in (0,1)$.  Then,  $K(\cdot,t) \in L^1(\R) \cap C^1(\R) \cap W^{1,\infty}(\R)$.  
In addition, $\partial_x K(\cdot,t)(x) = t^{-1} \big(\partial_x K(\cdot,1) * 
G(\cdot, 1-t^{1/3})\big)(t^{-1/2}x)$ for all $x \in \R$,  
$\partial_x K(\cdot,t) \in L^1(\R) \cap C_{\infty}(\R)$, and
\begin{equation}\label{eq:Kernel-Kd-0}
  \|\partial_x K(\cdot,t) \|_{L^1} \le \frac{C}{\sqrt t}
  \big[(1-t^{1/3})^{3/4}+ (1-t^{1/3})^{-3/4}  \big], 
\end{equation}
where $C$ is a positive constant independent of $t$.
\end{Lemma}
\begin{proof}
Let $f$ denote the function given by $f(\xi):= \xi \hat K(\xi,1)$ 
for all $\xi \in \R$.  Then  $f \in C_\infty(\R)\cap C^2(\R)$.  Furthermore,
\begin{equation*}
  f'(\xi)= \Big[1-\xi\Big(2\xi-\frac43 \xi^{1/3} \Phi(\xi)  
  \Big)  \Big]  \hat K (\xi,1), 
  \;\;\; \text{for} \;\; \xi \not = 0, 
\end{equation*}
and
\begin{equation*}
  f''(\xi)= \Big\{ \Big[-4\xi+\frac{16}{9} \xi^{1/3} \Phi(\xi)\Big]
  -\Big[1-\xi\Big(2\xi-\frac43 \xi^{1/3} \Phi(\xi)\Big)  \Big]  
   \Big[2\xi-\frac43\xi^{1/3} \Phi(\xi)  \Big] \Big\} \hat K (\xi,1),  
\end{equation*} 
for $\xi \not = 0$. By using Lemma \ref{Lemma:calculus2}, it follows that 
$\|f \|_{L^1} \le C$, $\|f' \|_{L^1} \le C$, and $\|f'' \|_{L^1} \le C$.  
Moreover, $|f' (\xi)| \rightarrow 0$ as $|\xi| \rightarrow + \infty$, and 
$f'(0^+)=1=f'(0^-)$.  Thus, Lemma \ref{Lemma:calculus} implies that 
$\partial_x K(\cdot,1) \in C_{\infty}(\R) \cap L^1(\R)$.   Therefore, 
using Lemma \ref{Lemma:Kernel-K}, we have that $K(\cdot,1)\in W^{1,1}(\R)$. 
Let $t\in (0,1)$.  Lemma \ref{Lemma:calculus1} implies that
\begin{equation}\label{eq:Kernel-Kd-0b}
  \| G(\cdot,1-t^{1/3}) \|_{L^\infty} \le C \cdot (1-t^{1/3})^{-{3/4}}.
\end{equation}
Thus, applying Lemma \ref{Lemma:calculus3} to 
equation (\ref{eq:Kernel}), taking into account (\ref{eq:Kernel-Kd-0b}), 
one sees that $K(\cdot,t) \in C^1(\R)\cap W^{1,\infty}(\R)$, and 
$\partial_x K(\cdot,t)(x)=t^{-1}\big( \partial_xK(\cdot,1) 
*G(\cdot,1-t^{1/3})\big)(t^{-1/2}x)$ for all $x \in \R$.  Furthermore, 
\begin{eqnarray*}
  \| \partial_x K(\cdot,t) \|_{L^1} 
  &=& t^{-1/2} \int |\partial_x K(\cdot,1) * G(\cdot,1-t^{1/3}) |(y)dy \\
  &\le& Ct^{-1/2}\big[(1-t^{1/3})^{3/4} + (1-t^{1/3})^{-3/4}  \big],
\end{eqnarray*}
where in the last step we have used Young's inequality and 
Lemma \ref{Lemma:Kernel-G}.
\end{proof}
Next lemma will be useful to study $\|\partial_x K(\cdot, t)\|_{L^1}$ for 
$t \ge t_0$, where $t_0>0$.
\begin{Lemma}\label{Lemma:Kernel-Kd-00} 
Suppose that $t>0$.  Then, $\partial_x K(\cdot,t) \in L^1(\R) \cap C_{\infty}(\R)$, 
and 
\begin{equation}\label{eq:Kernel-Kd-00}
 \|\partial_x K(\cdot, t)\|_{L^1}  \le C\Big[ \frac{1}{t} 
 +t^2  e^{\frac{4}{27}a^3t}\Big], 
\end{equation}
where $C$ is a positive constant independent of $t$.
\end{Lemma}
\begin{proof}
Let $t>0$.  For $\xi \in \R$, we define $h(\xi,t):=\xi \hat K(\xi,t)$.  Then 
\begin{equation*}
  \partial_\xi h(\xi,t) = \Big[ 1-t\Big(2\xi^2-\frac43 \xi^{\frac43} \Phi(\xi)\Big)  
  \Big] \hat K(\xi,t), 
  \;\;\;\; \text{for} \;\; \xi \not = 0,
\end{equation*}
and 
\begin{equation*}
  \partial_\xi^2 h(\xi,t) = -t\Big\{ 
  \Big[4\xi -\frac{16}{9}\xi^{\frac13} \Phi(\xi)  \Big]
  +\Big[ 1-t\Big(2\xi^2-\frac43 \xi^{\frac43} \Phi(\xi)\Big) \Big] 
  \Big[2\xi -\frac43 \xi^{\frac13}\Phi (\xi)  \Big]
  \Big\}\hat K(\xi,t), 
\end{equation*}
for $\xi \not = 0 $.   We see that 
$h(\cdot,t)\in C_{\infty}(\R) \cap C^2(\R)$, 
$|\partial_\xi h(\xi,t)|\rightarrow 0$ as $|\xi| \rightarrow + \infty$, and 
$\partial_\xi h (0^+,t)=1=\partial_\xi h(0^-,t)$.  Moreover, by Lemma 
\ref{Lemma:calculus2}, we have that
\begin{eqnarray*}
   \|h(\cdot,t)\|_{L^1} &\le& C\Big[\frac{1}{t} + e^{\frac{4}{27}a^3t} \Big], \\
   \|\partial_\xi h(\cdot,t)\|_{L^1} &\le& 
   C\Big[\frac{1}{t^{1/2}} + \frac{1}{t^{1/6}} +(1+t)e^{\frac{4}{27}a^3t} \Big],  
   \;\;\;\text{and} \\
   \|\partial_\xi^2 h(\cdot,t)\|_{L^1} &\le& 
   C\Big[1+t^{1/3}+t^{2/3} +(t+t^2)e^{\frac{4}{27}a^3t} \Big].
\end{eqnarray*}
The proof of the lemma is now completed by applying Lemma \ref{Lemma:calculus}.
\end{proof}
The next result provides a unified upper bound for 
$\|\partial_x  K(\cdot, t) \|_{L^1}$ for any $t>0$, which takes the best 
of the corresponding bounds obtained in Lemmas \ref{Lemma:Kernel-Kd-0} 
and \ref{Lemma:Kernel-Kd-00}.
\begin{Lemma}\label{Lemma:Kernel-Kd} 
Suppose that $t>0$. Then the function 
$\partial_x K(\cdot,t) \in L^1(\R) \cap C_{\infty}(\R)$.  
Moreover, 
\begin{equation}\label{eq:Kernel-Kd-1}
  \|\partial_x  K(\cdot, t) \|_{L^1} 
  \le C  \Big[\frac{1}{\sqrt t}+t^2 e^{\frac{4}{27}a^3t}\Big], 
\end{equation}
where $C$ is a positive constant independent of $t$.
\end{Lemma}
The following lemma will be used in the proof of Lemma \ref{Lemma:semig} 
below.
\begin{Lemma}\label{Lemma:unif}
\begin{equation}\label{eq:unif}
  \lim_{A \rightarrow +\infty } \int_{|y|>A}
  \big|K(\cdot,1) * G(\cdot, 1-h^{1/3})\big|(y) dy =0,  
  \;\;\; \text{ uniformly in } h \in [0,1).
\end{equation}
\end{Lemma}
\begin{proof}
We recall that  
\begin{equation*} 
  \hat K(\xi,1)=e^{-[\xi^2-\xi^{4/3} \Phi(\xi)]}, 
  \;\;\text{ and }  \;\;
  \hat G (\xi,1-h^{1/3}) = \frac{1}{\sqrt{2\pi}} 
  e^{-(1-h^{1/3}) \xi^{4/3} \Phi(\xi)}.
\end{equation*}  
From (\ref{eq:unif1}) we see that
\begin{eqnarray*}
   &&\int_{|y|>A} \big|K(\cdot,1) * G(\cdot, 1-h^{1/3})\big|(y) dy \\
   &&\le \Big( \int_{|y|>A} \frac{dy}{1+y^2} \Big)^{\frac12}
   \Big(\int (1+y^2) \big|K(\cdot,1) * G(\cdot, 1-h^{1/3})(y)\big|^2 
   dy \Big)^{\frac12}  \\
   &&\le  \frac{C}{\sqrt A} \; \| \hat K(\cdot,1) 
     \hat G(\cdot, 1-h^{1/3}) \|_{H^1} \le \frac{C'}{\sqrt A}.
\end{eqnarray*}
This concludes the proof.
\end{proof}

The next lemma shows that $(E(t))_{t \ge0}$ is a $C^0$-semigroup on the 
Banach space $Y$ and also on the Banach space $X$.
\begin{Lemma}\label{Lemma:semig}
{\bf{i.)}}  If $u_0 \in C_b(\R)$, then $u(t):=E(t)u_0 \in C_b(\R)$ for every  
$t\ge0$.  In addition,  
 \begin{equation}\label{eq:semig}
    \|E(t)\|_{\L (C_b(\R))} \le C \cdot \big(1 +t^2 
    e^{\frac{4}{27}a^3t} \big), \;\;\; \text{ for all } t>0.
 \end{equation}
Moreover, $(E(t))_{t \ge0}$ is a $C^0$-semigroup on $Y$.  

\noindent
{\bf{ii.)}}  $(E(t))_{t \ge0}$ is a $C^0$-semigroup on $X$.  Furthermore, 
(\ref{eq:semig}) remains true if the space $C_b(\R)$ is replaced by $C_b^1(\R)$.
\end{Lemma}
\begin{proof}
{\bf{i.)}} Let $u_0 \in C_b(\R)$, and $t>0$.  Since 
$u(x,t)=E(t)u_0(x)=\frac{1}{\sqrt{2\pi}} \int K(x-y,t)u_0(y)dy$, it 
follows that 
\begin{equation*}
  \|u(t)\|_{L^\infty} \le \frac{1}{\sqrt{2\pi}} \|u_0\|_{L^\infty} 
  \|K(\cdot,t)\|_{L^1}
  \le C \cdot \big(1+ t^2 e^{\frac{4}{27}a^3t}  
  \big) \|u_0\|_{L^\infty},  
\end{equation*}
where the last inequality is a consequence of Lemma 
\ref{Lemma:Kernel-K}.  Moreover,
\begin{equation*}
  |u(x+h,t)-u(x,t)| \le \frac{\|u_0 \|_{L^\infty}}{\sqrt{2\pi}}
  \|K(\cdot+h,t) -K(\cdot,t)\|_{L^1} \;\; \rightarrow 0 
  \text{ as } h \rightarrow 0.
\end{equation*}
Thus, we have proved that if $u_0 \in C_b(\R)$, then $u(t) \in C_b(\R)$ 
for all $t\ge0$.  In addition, one can see that  
$E(t+s) \phi = E(t) E(s) \phi$, for all $t,s \ge 0$, and $\phi \in C_b(\R)$. 

Suppose now that $t=0$, $u_0 \in Y \setminus \{ 0\}$,  and  $h \in (0,1)$.  
Since $\hat K(0,h) = \frac{1}{\sqrt{2\pi}} \int K(z,h)dz=1$, 
and using (\ref{eq:Kernel}) we have that
\begin{eqnarray}\label{eq:semig1}
  && |u(x,h)- u_0(x)|= \frac{1}{\sqrt{2\pi}} \Big| \int K(z,h) \big( 
   u_0(x-z)-u_0(x)\big) dz \Big|  \nonumber \\
  && = \frac{1}{\sqrt{2\pi}} \Big|\int h^{-1/2} \big( 
  K(\cdot,1) * G(\cdot, 1-h^{1/3}) \big)(h^{-1/2}z) 
  \big(u_0(x-z)-u_0(x)\big) dz \Big| \nonumber \\
  && = \frac{1}{\sqrt{2\pi}} \Big|\int \big(
  K(\cdot,1) * G(\cdot, 1-h^{1/3}) \big)(y) 
  \big(u_0(x-y \sqrt h)-u_0(x) \big) dy \Big|.
\end{eqnarray}
Let $\epsilon >0$.  By Lemma \ref{Lemma:unif}, there exists $A>0$, 
such that for every $h \in [0,1)$,
\begin{equation}\label{eq:semig2}
   \int_{|y|>A} \big|K(\cdot,1) * G(\cdot, 1-h^{1/3})\big|(y) dy 
   < \frac{\sqrt{2\pi} \epsilon}{4\|u_0\|_{L^\infty}}.
\end{equation}
Since $u_0$ is uniformly continuous, there exists $\delta>0$ such 
that for all $z,w \in \R$, 
\begin{equation}\label{eq:semig3}
 \text{if } |z-w|<\delta, \text{ then } 
 |u_0(z)-u_0(w)|<\sqrt{2\pi} \epsilon/(2C \|K(\cdot,1)\|_{L^1}),
\end{equation} 
where $C$ is a positive constant such that  
$\|G(\cdot, 1-h^{1/3})\|_{L^1} \le C$ for all $h \in [0,1/2)$ 
(see Lemma \ref{Lemma:Kernel-G}).  Let  
$h\in(0, \min\{\frac12,\frac{\delta^2}{A^2}  \})$.  Using 
(\ref{eq:semig1})-(\ref{eq:semig3}), we get
\begin{eqnarray*}
  &&|u(x,h)-u_0(x)| \le \frac{2\|u_0\|_{L^\infty}}{\sqrt{2\pi}} 
    \int_{|y|>A} \big|K(\cdot,1) * G(\cdot, 1-h^{1/3}) (y) \big| dy \\
  &&+\frac{1}{\sqrt{2\pi}} \int_{|y| \le A} \big|
  K(\cdot,1) * G(\cdot, 1-h^{1/3})(y)\big| \;\; 
  \big|u_0(x-y \sqrt h)-u_0(x) \big| dy \\
  &&\le \frac{\epsilon}{2} +\frac{\epsilon}{2C\|K(\cdot,1) \|_{L^1}}
    \| K(\cdot,1) * G(\cdot, 1-h^{1/3}) \|_{L^1} \le \epsilon, 
\end{eqnarray*}
for all $x \in \R$, where the last inequality is a consequence 
of Young's inequality.  Therefore,   
\begin{equation}\label{eq:semig4}
  \lim_{h \rightarrow 0} \|u(h) - u_0\|_{L^\infty} =0.
\end{equation}

We notice that if $u_0\in Y$, then $u(t)=E(t)u_0 \in Y$ for all $t\ge0$.  
In fact, assume $t>0$ and let $\epsilon>0$ be given.   Since $u_0$ is uniformly 
continuous, there exists $\delta>0$ such that if $|h| <\delta$, then 
$|u_0(x+h)-u_0(x)| < \epsilon \sqrt{2 \pi}/ \|K(\cdot,t)\|_{L^1}$, for 
any $x\in \R$.   Suppose $|h|<\delta$, then 
\begin{equation*}
  |u(x+h,t)-u(x,t)| \le \frac{1}{\sqrt{2\pi}} 
  \int \big| K(y,t) \big| \big|u_0(x-y+h)-u_0(x-y)\big|dy 
  \le \epsilon, \;\;  \text{ for all } x \in \R.
\end{equation*}
Hence, $u(t)$ is uniformly continuous, for all $t>0$.  

Assume now that $t>0$, and $u_0 \in Y$.  It follows from (\ref{eq:semig4}) and 
the semigroup property that 
\begin{equation*}
  \lim_{h \downarrow 0} \|E(t+h) u_0 -E(t) u_0 \|_{L^\infty} =0.
\end{equation*}
On the other hand, for $h>0$, one can see that
\begin{eqnarray*}
  |u(x,t-h)-u(x,t)| &=& \frac{1}{\sqrt{2\pi}} 
  \Big| \int K(x-y,t-h) \big(u_0(y) -u(y,h) \big) dy \Big| \\
  &\le& \frac{1}{\sqrt{2\pi}}  \|K(\cdot,t-h)\|_{L^1} 
  \|u(h)-u_0\|_{L^\infty} \\
  &\le& C\cdot \big(1 +(t-h)^2 e^{\frac{4}{27}a^3(t-h)}  \big)  
  \|u(h)-u_0\|_{L^\infty},
\end{eqnarray*}
where the last inequality is a consequence of Lemma \ref{Lemma:Kernel-K}.  
Equation (\ref{eq:semig4}) and the last inequality imply that 
\begin{equation*}
  \lim_{h \downarrow 0} \|E(t-h) u_0 -E(t) u_0 \|_{L^\infty} =0.
\end{equation*}

{\bf{ii.)}}  Let $u_0 \in X$.  By item {\bf{i.)}} above, we already know 
that $u\in C([0,+\infty);Y)$, where $u(t)=E(t)u_0$ for all $t\ge0$.  We will 
now prove that $\partial_xu(t) \in C_b(\R)$, $\partial_xu(t)$ is uniformly continuous,  
and $\lim_{h \rightarrow 0} \|\partial_x u(t+h) -\partial_x u(t)\|_{L^\infty} =0$, 
for all $t\ge 0$. Suppose first that $t>0$.  Then
\begin{eqnarray*}
  &&\Big| \frac{u(x+h,t)-u(x,t)}{h} -\frac{1}{\sqrt{2\pi}} 
  K(\cdot,t)*u_0'(x)\Big| \\
  &&=\frac{1}{\sqrt{2\pi}} \Big|\Big(K(\cdot,t)*
  \frac{u_0(\cdot+h)-u_0(\cdot)}{h}  \Big)(x) -K(\cdot,t) *u_0'(x)\Big| \\
  && \le \frac{1}{\sqrt{2\pi}} \int \big|K(x-y,t)\big|  
  \Big|\frac{u_0(y+h)-u_0(y)}{h} -u_0'(y) \Big| dy.
\end{eqnarray*}
The dominated convergence theorem and Lemma \ref{Lemma:Kernel-K} imply 
that the last expression tends to zero as $h$ goes to zero.  Then there exists
\begin{equation}\label{eq:semig5}
  \partial_x u(x,t) = \frac{1}{\sqrt{2\pi}}  K(\cdot,t)*u_0'(x),  
  \;\; \text{for all } x \in \R, \text{ and } t>0.
\end{equation}
It is easy to see that the last expression is also valid if we only require 
that $u_0 \in C_b^1(\R)$.  It follows from (\ref{eq:semig5}) and 
Lemma \ref{Lemma:Kernel-K} that 
\begin{equation*}
  \|\partial_x u(\cdot,t) \|_{L^\infty}   \le C \cdot 
  \big(1+t^2 e^{\frac{4}{27}a^3t} \big) \|u_0'\|_{L^\infty}, 
  \;\;  \text{for all } t>0.
\end{equation*}
Using the fact that $u_0'$ is uniformly continuous, (\ref{eq:semig5}), and 
Lemma \ref{Lemma:Kernel-K}, it follows that 
$\partial_x u(\cdot,t)$ is uniformly continuous, for all $t>0$.

Finally, since $(E(t))_{t\ge0}$ is a $C^0$-semigroup on the space $Y$ and 
using (\ref{eq:semig5}), we see that  
$\lim_{h \rightarrow 0} \|\partial_x u(t+h) -\partial_x u(t)\|_{L^\infty} =0$, 
for all $t\ge 0$.
\end{proof}



\subsection{Local Theory in the Space $X$}\label{subsection:local}
In this Sub-section we will use the Banach fixed-point theorem on an 
appropriate complete metric space to find a local-in-time solution to 
the integral equation associated to the IVP (\ref{eq:IVP}).  The following 
lemma will be helpful during the proof of Theorem \ref{T:local} below.
\begin{Lemma}\label{Lemma:int}
Suppose that $u \in C([0,T];X)$.  We define 
\begin{equation}\label{eq:int}
  D(\cdot,t):= \int_0^t K(\cdot,t-s) * \frac12 \partial_x u^2(\cdot,s) ds, 
  \;\;\; \text{ for } t \in [0,T].
\end{equation}
Then $D \in C([0,T];X)$.
\end{Lemma}
\begin{proof}
{\bf{i.)}} Let $t\in(0,T]$.  Now we first prove that $D(t) \in X$ .  In fact, 
\begin{eqnarray*}
  \|D(\cdot,t) \|_{L^\infty} &\le& \sup_{s\in[0,T]} \|u(\cdot,s)\|_{C_b^1}^2 
  \int_0^t \|K(\cdot,t-s)\|_{L^1} ds \\
  &\le& C \|u\|_{C([0,T];X)}^2 
  \int_0^t \big(1 + (t-s)^2  e^{\frac{4}{27}a^3(t-s)}\big) ds,
\end{eqnarray*}
where the last inequality is a consequence of Lemma \ref{Lemma:Kernel-K}.  Then
\begin{equation}\label{eq:int1}
  \|D(\cdot,t) \|_{L^\infty} \le C' \; \|u\|_{C([0,T];X)}^2 \; \nu(t),
\end{equation}
where 
\begin{equation}\label{eq:int1a}
  \nu(r):= r + r^2 e^{\frac{4}{27}a^3r}, \;\; \text{ for all  } r \ge 0.  
\end{equation}
Moreover, 
\begin{equation*}
  \|D(\cdot+h,t) -D(\cdot,t)\|_{L^\infty} 
  \le \int_0^t \| K(\cdot, t-s) \|_{L^1} 
  \|\frac{1}{2}\partial_x u^2(\cdot+h,s) - 
  \frac{1}{2}\partial_x u^2(\cdot,s)\|_{L^\infty} ds.
\end{equation*}
Using the fact that $\partial_x u(\cdot,s) u(\cdot,s)$ is uniformly 
continuous on $\R$, for all $s \in [0,T]$, Lemma \ref{Lemma:Kernel-K}, 
and the dominated convergence theorem, it follows from the last inequality that 
$D(t)$ is uniformly continuous on $\R$. Now we claim that there exists 
$\frac{\partial D}{\partial x} (x,t)$ (in the classical sense), for all $x \in \R$, and 
\begin{eqnarray}\label{eq:int2}
  \frac{\partial D}{\partial x} (x,t) 
  &=& \int_0^t \partial_x \big( K(\cdot, t-s)*\frac12 \partial_x u^2(\cdot,s) \big)(x) ds 
  \nonumber \\
  &=& \int_0^t \partial_x K(\cdot,t-s) * \frac12 \partial_x u^2(\cdot,s) (x) ds, 
  \;\; \text{ for all } x \in \R.
\end{eqnarray}
We now establish the last claim.  It follows from Lemmas 
\ref{Lemma:Kernel-K},  \ref{Lemma:Kernel-Kd}, and \ref{Lemma:calculus3} that 
$K(\cdot,t-s) *\frac12 \partial_x u^2(\cdot,s) \in C^1(\R) \cap W^{1,\infty} (\R)$, 
and 
\begin{equation}\label{eq:int3}
  \partial_x \big(K(\cdot,t-s) *\frac12 \partial_x u^2(\cdot,s)  \big)(x)
  =\big(\partial_x K(\cdot,t-s) *\frac12 \partial_x u^2(\cdot,s) \big) (x),
\end{equation}
for all $x \in \R$ and $s \in [0,t)$.  Moreover, 
\begin{eqnarray}\label{eq:int4}
 \!\!\!\!\!\!\!\!\!\!\!\!\!\!\!\!\!\!\!&&
 \Big|\frac{D(x+h,t)-D(x,t)}{h} - \int_0^t \partial_x K(\cdot,t-s) 
 *\frac12 \partial_x u^2(\cdot,s) (x) ds\Big|   \nonumber \\
 \!\!\!\!\!\!\!\!\!\!\!\!\!\!\!\!\!\!\!&&
 = \Big| \int_0^t\Big(\frac{K(\cdot+h,t-s)-K(\cdot,t-s)}{h} 
 -\partial_xK(\cdot,t-s)\Big)*\frac12 \partial_x u^2(\cdot,s) (x) ds \Big|. 
\end{eqnarray}
In addition,
\begin{eqnarray}\label{eq:int5}
 && \Big{\|}\Big(\frac{K(\cdot+h,t-s)-K(\cdot,t-s)}{h} 
    -\partial_xK(\cdot,t-s)\Big)*\frac12 \partial_x u^2(\cdot,s)  
    \Big{\|}_{L^\infty} \nonumber \\
 && \le \|u \|_{C([0,T];X)}^2  \;\;
    \Big{\|}\frac{K(\cdot+h,t-s)-K(\cdot,t-s)}{h} 
    -\partial_xK(\cdot,t-s)\Big) \Big{\|}_{L^1} \nonumber  \\
 && \le \|u \|_{C([0,T];X)}^2 \;\;
    \Big( \frac{1}{|h|} \Big|\int_0^h \|\partial_xK(\cdot+y,t-s)\|_{L^1}dy   
    \Big| + \| \partial_x K(\cdot,t-s)  \|_{L^1} \Big)  \nonumber \\
 && \le C \|u \|_{C([0,T];X)}^2 \;\; 
    \Big( \frac{1}{\sqrt{t-s}} +(t-s)^2 
    e^{\frac{4}{27}a^3(t-s)} \Big) \in L^1((0,t),ds),
\end{eqnarray}
where the last inequality is a consequence of Lemma \ref{Lemma:Kernel-Kd}.  
The claim now follows from (\ref{eq:int3})-(\ref{eq:int5}) and the 
dominated convergence theorem.  \\
It follows directly from (\ref{eq:int2}) and Lemma \ref{Lemma:Kernel-Kd} that 
\begin{equation}\label{eq:int6}
  \|\partial_x D(\cdot,t) \|_{L^\infty}
  \le C \; \|u \|_{C([0,T];X)}^2  \; \mu(t),
\end{equation}
where 
\begin{equation}\label{eq:int7a}
  \mu(r):= \sqrt r + r^2 e^{\frac{4}{27}a^3r}, \;\; \text{ for all  } r \ge 0.  
\end{equation}
The fact that $\partial_x D(\cdot,t)$ is uniformly continuous on $\R$ can be 
shown similarly to the analogous result for $D(\cdot,t)$, using 
Lemma \ref{Lemma:Kernel-Kd} instead of Lemma \ref{Lemma:Kernel-K}. 

{\bf{ii.)}} We will now prove that $D \in C([0,T];X)$.  Let $t \in [0,T)$.  
We first assume that $h>0$.  Then 
\begin{equation*}
 \|D(\cdot,t+h) -D(\cdot,t) \|_{L^\infty}
 \le I_1(t,h) +I_2(t,h),
\end{equation*}
where 
\begin{eqnarray*}
  I_1(t,h)&:=& \int_0^t \big{\|} \big( K(\cdot,t+h-s) -K(\cdot,t-s) \big) 
  *\frac12 \partial_x u^2(\cdot,s)\big{\|}_{L^\infty} ds , 
  \;\;\text{ and} \\ 
  I_2(t,h)&:=& \int_t^{t+h} \big{\|} K(\cdot,t+h-s)  
  *\frac12 \partial_x u^2(\cdot,s)\big{\|}_{L^\infty} ds.
\end{eqnarray*}
We see that
\begin{eqnarray*}
  && \big{\|} \big( K(\cdot,t+h-s) -K(\cdot,t-s) \big) *\frac12 
     \partial_x u^2(\cdot,s)\big{\|}_{L^\infty} \\
  && \le \|u \|_{C([0,T];X)}^2 \;\;
     \| K(\cdot,t+h-s) -K(\cdot,t-s) \|_{L^1} \\
  && \le C \|u \|_{C([0,T];X)}^2 \;\;
     \big(1+ T^2 e^{\frac{8}{27}a^3T}\big)
     \;\; \in L^1((0,t),ds),
\end{eqnarray*}
where the last inequality follows from Lemma \ref{Lemma:Kernel-K} and 
the fact that $h \in (0,T)$.  Thus, using Lemma \ref{Lemma:semig} and the 
dominated convergence theorem we have that 
\begin{equation*}
  I_1(t,h)= \sqrt{2\pi} \int_0^t  \big{\|}\big(E(h)-1\big) E(t-s) \frac12 
  \partial_x u^2(\cdot,s) \big{\|}_{L^\infty} ds \rightarrow 0,  
  \;\; \text{ as } h \downarrow 0.
\end{equation*}  
Moreover, using Lemma \ref{Lemma:Kernel-K} we get like in (\ref{eq:int1}) that 
\begin{eqnarray*}
  I_2(t,h) \le  C \; \|u \|_{C([0,T];X)}^2 \; \nu(h) \rightarrow 0,  
  \;\; \text{ as } h \downarrow 0,
\end{eqnarray*}
where $\nu(\cdot)$ is given by (\ref{eq:int1a}).  Hence, 
\begin{equation}\label{eq:int7}
  \lim_{h \downarrow 0} \|D(\cdot,t+h)-D(\cdot,t) \|_{L^\infty} =0.
\end{equation}

On the other hand, it follows from (\ref{eq:int2}) that
\begin{equation*}
  \|\partial_xD(\cdot,t+h) -\partial_xD(\cdot,t) \|_{L^\infty}
 \le J_1(t,h) +J_2(t,h),
\end{equation*}
where 
\begin{eqnarray*}
  J_1(t,h)&:=& \int_0^t \big{\|} \big( \partial_xK(\cdot,t+h-s) 
  -\partial_xK(\cdot,t-s) \big) 
  *\frac12 \partial_x u^2(\cdot,s)\big{\|}_{L^\infty} ds , 
  \;\;\text{ and} \\ 
  J_2(t,h)&:=& \int_t^{t+h} \big{\|} \partial_xK(\cdot,t+h-s)  
  *\frac12 \partial_x u^2(\cdot,s)\big{\|}_{L^\infty} ds.
\end{eqnarray*}
It follows directly from Lemma \ref{Lemma:Kernel-Kd} that 
\begin{eqnarray*}
  J_2(t,h) &\le&  \|u \|_{C([0,T];X)}^2 \;\; \int_0^h 
  \|\partial_x K(\cdot,\tau) \|_{L^1} d\tau  \\
  &\le& C \; \|u \|_{C([0,T];X)}^2 \; \mu(h)  \rightarrow 0, 
  \;\; \text{ as } h \downarrow 0,
\end{eqnarray*}
where $\mu(\cdot)$ is given by (\ref{eq:int7a}).  
To estimate $J_1(t,h)$ we first extend $\partial_xK$ for all times in the 
following way:
\begin{equation*}
  H(\cdot,s) := \left\{
  \begin{array}
  [c]{l}
  \partial_x K(\cdot,s),\text{ if  } s \in [0,T],\\
  0, \text{ if  } s \in \R \setminus [0,T].
  \end{array}
\right.
\end{equation*}
We note that $H \in L^1(\R^2)$.  In fact, by Lemma \ref{Lemma:Kernel-Kd} 
we get
\begin{equation*}
  \int\int |H(x,s)|dxds = \int_0^T \|\partial_x K(\cdot,s) \|_{L^1}ds 
  \le C \; \mu(T).
\end{equation*}
Then
\begin{eqnarray*}
  J_1(t,h) 
  &\le& \|u \|_{C([0,T];X)}^2 \;\; \int_0^t \| \partial_x K(\cdot,\tau+h)
  - \partial_xK(\cdot,\tau) \|_{L^1} d\tau \\
  &\le& \|u \|_{C([0,T];X)}^2 \;\; \int\int |H(x,\tau+h)-H(x,\tau)|dx d\tau 
  \rightarrow 0, \;\; \text{ as } h \downarrow 0,
\end{eqnarray*}
where the last assertion follows from the continuity of translations in 
$L^1(\R^2)$.  Hence,
\begin{equation}\label{eq:int8}
  \lim_{h \downarrow 0} \|\partial_x D(\cdot,t+h)-
  \partial_xD(\cdot,t) \|_{L^\infty} =0.
\end{equation}
It follows from (\ref{eq:int7}) and (\ref{eq:int8}) that 
$ \lim_{h \downarrow 0} \|D(\cdot,t+h)-D(\cdot,t) \|_{C_b^1} =0$.

The case when $t \in (0,T]$ and $h<0$ can be shown similarly to 
the previous case.  This finishes the proof of the lemma.
\end{proof}

The next theorem is the main result of this section, it states 
local-in-time existence of the solution of the integral equation 
associated to the IVP (\ref{eq:IVP}).
\begin{Theorem}\label{T:local}
Suppose $u_0 \in X$.  Then there exist $T=T(\|u_0\|_{C_b^1})>0$ and a 
unique function $u \in C([0,T];X)$ satisfying the integral equation
\begin{equation}\label{eq:local}
 u(\cdot,t) = E(t)u_0(\cdot) - \frac12 \int_0^t 
 E(t-s) \partial_x u^2(\cdot,s)ds,
\end{equation}
where $E(t)$ is defined by (\ref{eq:ope}).
\end{Theorem}
\begin{proof}
Let $M:=1+2\|u_0\|_{C^1_b}$.  Let $T>0$ be fixed.  $T$ will be suitably 
chosen later.  We now consider the nonlinear operator $A$ given by
\begin{equation*}\label{eq:ope1}
  (Af)(\cdot,t) := E(t)u_0(\cdot) -\frac12 \int_0^t E(t-s) 
  \partial_x f^2(\cdot,s)ds,
\end{equation*}
 defined on the complete metric space
\begin{equation*}\label{eq:ope2}
  \Theta_T^M :=\Big{\{}f \in C([0,T];X) ; \sup_{t\in[0,T]} 
  \|f(\cdot,t) \|_{C_b^1} \le M\Big{\}}.
\end{equation*}

Let $f \in \Theta_T^M$.  It follows from Lemmas \ref{Lemma:semig} 
and \ref{Lemma:int} that $Af \in C([0,T];X)$.  

We will now prove that we can choose $T=\tilde T>0$ small enough such that 
$A(\Theta_{\tilde T}^M) \subset \Theta_{\tilde T}^M$.  
Suppose $f \in \Theta_T^M$.  By Lemma \ref{Lemma:semig} we know that 
$\lim_{h \downarrow 0} \|(E(h)-1) u_0 \|_{C_b^1}=0$.  Then there exists 
$\delta = \delta (\|u_0\|_{C_b^1})>0$ such that if 
$0 \le h \le \delta$, then $\|E(h)u_0\|_{C^1_b} \le \frac12 \big( 1
+ 3\| u_0 \|_{C_b^1}\big)$.  If $T \le \delta$, using 
Lemmas \ref{Lemma:Kernel-K} and \ref{Lemma:Kernel-Kd}, and 
(\ref{eq:int2}), we get 
\begin{eqnarray*}
  && \|(Af)(\cdot,t) \|_{C_b^1} \le \frac12 \big( 1+3\|u_0\|_{C_b^1}\big) \\
  && +\frac{1}{2\sqrt{2\pi}} \int_0^t \big(\|K(\cdot,t-t')\|_{L^1}  
     +\|\partial_x K(\cdot,t-t') \|_{L^1}\big) \|f\|_{C([0,T];X)}^2 dt' \\
  && \le \frac12 \big( 1+3\|u_0\|_{C_b^1}\big) + M^2C 
     \int_0^t\Big[\frac{1}{\sqrt \tau} +\tau^2 e^{\frac{4}{27}a^3\tau}\Big] d\tau \\
  && \le \frac12 \big( 1+3\|u_0\|_{C_b^1}\big) + M^2C \; \mu(T),
\end{eqnarray*}
for all $t \in [0,T]$, where $\mu(\cdot)$ is given by (\ref{eq:int7a}).  
Take $T^\dagger>0$ such that 
$M^2C \; \mu(T^\dagger) \le \frac12\big( 1+\|u_0\|_{C_b^1} \big)$.  Thus, if 
$\tilde T \in (0, \min \{\delta, T^\dagger \}]$, then 
$\|(Af)(\cdot,t) \|_{C_b^1} \le M$ for all $t\in [0,\tilde T]$.

Finally, we will prove that there exists $T' \in (0,\tilde T]$ such 
that $A$ is contractive on $\Theta_{T'}^M$.  Suppose that 
$f,g \in \Theta_{\tilde T}^M$.  Let $t\in[0,\tilde T]$.  Then
\begin{eqnarray*}
 && \|(Af)(\cdot,t) - (Ag)(\cdot,t) \|_{C_b^1} \\
 && \le C \int_0^t \big(\|K(\cdot,t-t')\|_{L^1}  
     +\|\partial_x K(\cdot,t-t') \|_{L^1}\big) \|\partial_x f^2(\cdot,t')  
     -\partial_xg^2(\cdot,t')\|_{L^\infty} dt' \\
 && \le C \int_0^t \big(\|K(\cdot,t-t')\|_{L^1}  
     +\|\partial_x K(\cdot,t-t') \|_{L^1}\big)  \\
 && \times \big[ \|f(\cdot,t') \|_{L^\infty} 
    \|\partial_x(f(\cdot,t')-g(\cdot,t'))\|_{L^\infty}  
    +\|f(\cdot,t')-g(\cdot,t') \|_{L^\infty} 
     \|\partial_xg(\cdot,t') \|_{L^\infty}\big] dt' \\
 && \le C M \|f-g\|_{C([0,\tilde T];X)} \; \mu(t).
\end{eqnarray*}
Taking $T' \in (0,\tilde T]$ such that $CM \; \mu(T')<1$, it 
follows that  $A$ is a contraction on $\Theta_{T'}^M$.  Therefore, 
the mapping $A$ has a unique fixed point $u \in \Theta_{T'}^M$ which 
satisfies equation (\ref{eq:local}) with $T'=T'(\|u_0\|_{C_b^1})>0$.  
The uniqueness of the solution of equation (\ref{eq:local}) in the 
class $C([0,T'];X)$ is a consequence of Proposition \ref{Prop:dep} below.
\end{proof}

The next proposition shows the continuous dependance of the solutions 
of equation (\ref{eq:local}) on the initial data.
\begin{Proposition}\label{Prop:dep}
Suppose that $u,v \in C([0,T];X)$ are solutions of equation 
(\ref{eq:local}) with initial data $u_0, v_0 \in X$ respectively.  Then  
for all $t \in [0,T]$ we have 
\begin{equation}\label{eq:dep}
  \|u(\cdot,t) -v(\cdot,t) \|_{C_b^1} \le C
  e^{\alpha t} \|u_0-v_0\|_{C_b^1},
\end{equation}
where $C$ and $\alpha$ are positive constants depending on 
$T, \|u\|_{C([0,T];X)}$, and $\|v\|_{C([0,T];X)}$.
\end{Proposition}
\begin{proof}  Let $t \in [0,T]$.  We write 
$w(\cdot,t):=u(\cdot,t)-v(\cdot,t)$.   Then 
\begin{equation}\label{eq:dep1}
  \|w(\cdot, t) \|_{C_b^1} \le \|E(t)(u_0-v_0)\|_{C_b^1} 
  +\frac12 \Big{\|} \int_0^t E(t-t') \big(\partial_x u^2(\cdot,t') 
  -\partial_xv^2 (\cdot,t')\big) dt' \Big{\|}_{C_b^1}.
\end{equation}
It follows from Lemma \ref{Lemma:semig} that
\begin{equation}\label{eq:dep2}
  \|E(t)(u_0-v_0)\|_{C_b^1} 
  \le C' \|u_0-v_0\|_{C_b^1}, 
\end{equation}
where 
\begin{equation*}\label{eq:depC1}
  C'= C \cdot \big(1+T^2 e^{\frac{4}{27}a^3T}\big).
\end{equation*}
Moreover, by Lemmas \ref{Lemma:Kernel-K} and \ref{Lemma:Kernel-Kd}, we get
\begin{eqnarray}\label{eq:dep3}
  && \frac12 \Big{\|} \int_0^t E(t-t') \big(\partial_x u^2(\cdot,t') 
     -\partial_xv^2 (\cdot,t')\big) dt' \Big{\|}_{C_b^1} \nonumber \\
  && \le \frac{\|u\|_{C([0,T];X)}+\|v\|_{C([0,T];X)}}{\sqrt{2\pi}} \nonumber \\
  && \times \int_0^t \big(\|K(\cdot,t-t')\|_{L^1}  
     +\|\partial_x K(\cdot,t-t') \|_{L^1}\big) \|w(\cdot,t')\|_{C_b^1} dt' 
     \nonumber \\
  && \le C \cdot \big( \|u\|_{C([0,T];X)}+\|v\|_{C([0,T];X)} \big) 
     \nonumber \\
  && \times \int_0^t \Big[ 1+(t-t')^2  e^{\frac{4}{27}a^3(t-t')} +
     \frac{1}{\sqrt{t-t'}}  \Big] \|w(\cdot,t')\|_{C_b^1} dt' 
     \nonumber \\
  && \le \tilde C \int_0^t \frac{\|w(\cdot,t')\|_{C_b^1}}{\sqrt{t-t'}} dt',
\end{eqnarray}
where 
\begin{equation*}\label{eq:depC2}
  \tilde C := C \cdot (1+T^{5/2} e^{\frac{4}{27}a^3T}) 
  \; \big(\|u\|_{C([0,T];X)}+\|v\|_{C([0,T];X)}\big). 
\end{equation*}
Thus, it follows from (\ref{eq:dep1}), (\ref{eq:dep2}) and (\ref{eq:dep3}) that 
\begin{equation*}
  \|w(\cdot,t) \|_{C_b^1} \le C' \|u_0-v_0\|_{C_b^1} 
  + \tilde C \int_0^t \frac{\|w(\cdot,t')\|_{C_b^1}}{\sqrt{t-t'}} dt'.
\end{equation*}
Then
\begin{eqnarray*}
  && \|w(\cdot, t) \|_{C_b^1} \le C' \|u_0-v_0\|_{C_b^1} \\
  && + \tilde C \int_0^t \frac{1}{\sqrt{t-t'}}\Big[ C' \|u_0-v_0\|_{C_b^1} 
     +\tilde C \int_0^{t'}  
     \frac{\|w(\cdot,r)\|_{C_b^1}}{\sqrt{t'-r}}  dr \Big] dt' \\
  && \le C' (1+2\tilde C \sqrt T) \|u_0-v_0\|_{C_b^1} 
     +\tilde C^2 \int_0^t \int_r^{t} 
     \frac{\|w(\cdot,r)\|_{C_b^1}}{\sqrt{t-t'}\sqrt{t'-r}}  dt'dr \\
  && = C \|u_0-v_0\|_{C_b^1} 
     +\tilde C^2 B\big(\frac12,\frac12\big) \int_0^t  \|w(\cdot,r)\|_{C_b^1} dr, 
\end{eqnarray*}
where $B(\cdot,\cdot)$ denotes the beta function defined by 
$B(x,y):=\int_0^1 t^{x-1} (1-t)^{y-1} dt$, for $\Re (x), \Re (y) >0$.  The 
proposition now follows by applying Gronwall's inequality to the last expression. 
\end{proof}

\subsection{Future Work}\label{subsection:future}
Some interesting problems remain, though: the study of the global 
well-posedness for the IVP (\ref{eq:IVP}) with initial data 
belonging to the space $X$, and the nonlinear stability theory of 
the travelling-wave solution of equation (\ref{eq:fow0}).  These  
two problems will be addressed elsewhere.


\medskip
\noindent {\bf{Acknowledgements:}} The authors were supported by 
the ANR-France project COPTER (Conception, optimisation et prototypage 
d'ouvrages de lutte contre l'erosion en domaine littoral) under 
grant No. NT05-2\_42253.  The authors wish to thank 
R\'emi Carles (I3M-Universit\'e Montpellier 2) and Natha\"el Alibaud 
(D\'epartement de Math\'ematiques de Besan\c{c}on) for 
fruitful discussions.  



\begin{thebibliography}{99}
\bibitem{aai:aai}N. Alibaud, P. Azerad and D. Is\`ebe,  \textit{A non-monotone 
nonlocal conservation law for dune morphodynamics},  Preprint 
(\texttt{http://arxiv.org/abs/0709.3360}).
\bibitem{dsh:dsh} O. Dur\'an, V. Schw\"ammle and H.J. Herrmann, \textit{Breeding 
and solitary wave behavior of dunes}, Phys. Rev. E, {\bf{72}} (2005), 021308, 5 pp.
\bibitem{f0:f0} A.C. Fowler,  \textit{Dunes and drumlins},  in ``Geomorphological 
fluid mechanics" (eds. N.J. Balmforth and A. Provenzale),  Springer-Verlag,  
Berlin (2001), pp. 430--454.
\bibitem{f1:f1} A.C. Fowler,  \textit{Evolution equations for dunes and 
drumlins},  Mathematics and environment (Spanish) (Paris, 2002),  RACSAM Rev. 
R. Acad. Cienc. Exactas F\'{\i}s. Nat. Ser. A Mat.,  {\bf{96}}  (2002),  377--387.
\bibitem{f2:f2} A.C. Fowler, ``Mathematics and the environment," Lecture Notes 
(\texttt{http://www2.maths.ox.ac.uk/\~{}fowler/courses/mathenvo.html}). 
\bibitem{hs:hs} H.J. Herrmann and G. Sauermann, \textit{The shape of dunes},  
Physica A, {\bf{283}} (2000), 24--30.
\bibitem{j:j} R.S. Johnson, ``A modern introduction to the mathematical 
theory of water waves," Cambridge Texts in Applied Mathematics, Cambridge 
University Press, Cambridge, 2004.
\bibitem{r:r} W. Rudin, ``Functional Analysis,"  Second edition, International 
Series in Pure and Applied Mathematics, McGraw-Hill, Inc., New York, 1991.


\end{thebibliography}
\end{document}